\documentclass[11pt]{amsart}

\usepackage[english]{babel}
\usepackage[utf8]{inputenc}
\usepackage[charter,cal=cmcal]{mathdesign}
\usepackage[T1]{fontenc}
\usepackage[a4paper]{geometry}
\usepackage{bbm}
\usepackage{etoolbox}
\usepackage{mathtools}

\newcommand{\m}[1]{
\ifdefequal{#1}{1}
{\mathbbm{#1}}
{\mathbb{#1}}
}

\newcommand{\q}[1]{\mathscr{#1}}
\newcommand{\cal}[1]{\mathcal{#1}}

\newcommand{\ds}{\displaystyle}

\newcommand{\e}{\varepsilon}

\newcommand{\weak}{\rightharpoonup}

\newcommand{\supp}{supp}
\newcommand{\jap}[1]{\left\langle #1 \right\rangle}
\newcommand{\partere}{\mathrm{Re} }

\newcommand{\be}{\begin{equation}}
\newcommand{\ee}{\end{equation}}

\DeclareMathOperator{\sgn}{\mathrm{sgn}}
\DeclareMathOperator{\Ai}{\mathrm{Ai}}

\DeclareMathOperator{\loc}{\mathrm{loc}}

\renewcommand{\Re}{\mathop{\mathrm{Re}}}

\renewcommand{\le}{\leqslant}
\renewcommand{\ge}{\geqslant}

\theoremstyle{plain}
\newtheorem{thm}{Theorem}
\newtheorem*{thm*}{Theorem}
\newtheorem{prop}[thm]{Proposition}
\newtheorem{cor}[thm]{Corollary}
\newtheorem{lem}[thm]{Lemma}
\newtheorem{claim}[thm]{Claim}

\theoremstyle{definition}

\theoremstyle{remark}
\newtheorem{nb}[thm]{Remark}

\makeatletter
\def\blfootnote{\xdef\@thefnmark{}\@footnotetext}
\makeatother

\allowdisplaybreaks

\title{Self-similar dynamics for the modified Korteweg-de Vries equation}

\date{}
\author{Simão Correia, Raphaël Côte \and Luis Vega}

\subjclass[2010]{35Q53 (primary),  35C06, 35B40, 34E10}  

\thanks{Simão Cor\-reia is supported by Fun\-da\-ção para a Ciência e Tecnologia under the project UID/MAT/4561/ 2019. Raphaël Côte is supported by the Agence Nationale de la Recherche under the contract MAToS ANR-14-CE25-0009-0. Luis Vega is supported by an ERCEA Advanced Grant 2014 669689 - HADE, by the MEIC project MTM2014-53850-P and MEIC Severo Ochoa excellence accreditation SEV-2013-0323.}

\begin{document}

\begin{abstract}
\noindent We prove a local well posedness result for the modified Korteweg-de Vries equation in a critical space designed so that is contains self-similar solutions. As a consequence, we can study the flow of this equation around self-similar solutions: in particular, we give an asymptotic description of small solutions as $t \to +\infty$ and construct solutions with a prescribed blow up behavior as $t \to 0$.
\end{abstract}

\maketitle

\section{Introduction}

In this paper, we are interested in the dynamics near self-similar solutions for the  modified Korteweg-de Vries equation:
\begin{align} \label{mkdv} \tag{mKdV}
\partial_t u + \partial_{xxx}^3 u + \epsilon \partial_x (u^3) =0, \quad u: \m R_t \times \m R_x \to \m R.
\end{align}
The signum $\epsilon \in \{\pm 1 \}$ indicates whether the equation is focusing or defocusing. In our framework, $\epsilon$ will play no major role.

The \eqref{mkdv} equation enjoys a natural scaling: 
if $u$ is a solution then
\[ u_\lambda(t,x) := \lambda^{1/3} u(\lambda t, \lambda^{1/3} x) \]
is also a solution to \eqref{mkdv}. As a consequence, the self-similar solutions, which preserve their shape under scaling
\[ S(t,x) = t^{-1/3} V(t^{-1/3} x), \]
are therefore of special interest. 

Self-similar solutions play an important role indeed for the \eqref{mkdv} flow: they exhibit an explicit blow up behavior, and are also related with the long time description of solutions. Even for small and smooth initial data, solutions display a modified scattering where self-similar solutions naturally appear: we refer to Hayashi and Naumkin \cite{HN99,HN01}, which was revisited by Germain, Pusateri and Rousset \cite{GPR16} and Harrop-Griffiths \cite{HaGr16}.

Another example where self-similar solutions of  the \eqref{mkdv} equation are relevant is in the long time asymptotics of the so-called Intermediate Long Wave (ILW) equation. This equation occurs in the propagation of waves in a one-dimensional stratified fluid in two limiting cases. In the shallow water limit, the propagation reduces to the KdV equation, while in the deep water limit, it reduces to the so-called Benjamin-Ono equation. In a recent work, Bernal-Vilchis and Naumkin \cite{B_VN19}
study the large-time behavior of small solutions of the (modified) ILW, and they prove that in the so-called self-similar region the solutions tend at infinity to a self-similar solution of \eqref{mkdv}. 

Self-similar solutions and the \eqref{mkdv} flow are also related to some other simplified models in fluid dynamics. More precisely, Goldstein and Petrich \cite{GP92} find a formal connection between the evolution of the boundary of a vortex patch in the plane under Euler equations and a hierarchy of completely integrable dispersive equations. The first element of this hierarchy is:
\[ \partial_t z = - \partial_{sss} z +\partial_s \bar z (\partial_{ss} z)^2, \quad |\partial_s z|^2=1, \]
where $z=z(t,s)$ is complex valued and parametrize by its arctlength $s$ a plane curve which evolves in time $t$. A direct computation shows that its curvature solves the focusing \eqref{mkdv} (with $\epsilon=1$), and self-similar solutions with initial data 
\begin{gather} \label{def:CI}
U(t) \weak c \delta_{0} + \alpha \mathop{\mathrm{v.p.}} \left( \frac{1}{x} \right) \quad \text{as } t \to 0^+, \quad \alpha, c \in \m R, \end{gather}
correspond to logarithmic spirals making a corner, see \cite{PV07}.

Finally \eqref{mkdv} is a member of a two parameter family of geometric flows that appears as a model for the evolution of vortex filaments. In this case, the filaments are curves that propagate in 3d, and their curvature and torsion determined a complex valued function that satisfies a non-linear dispersive equation. This equation, that depends on the two free parameters, is a combination of a cubic non-linear Schrödinger equation (NLS) and a complex modified Korteweg-de Vries equation. 

The particular case of cubic (NLS) has received plenty of attention. The corresponding geometric flow is known as either the binormal curvature flow or the Localized Induction Approximation, name that is more widely used in the literature in fluid dynamics. In this setting, the relevant role played by the self-similar solutions, including also logarithmic spirals, has been largely studied. We refer the reader to the recent paper by Banica and Vega \cite{BV18} and the references there in. Among other things, in this article the authors prove that the self-similar solutions have finite energy, when the latter is properly defined. Moreover, they give a well-posedness result in an appropriately chosen space of distributions that contains the self-similar solutions.

\bigskip

Our goal in this paper is to continue our work initiated in \cite{CCV19},  and to study the \eqref{mkdv} flow in spaces in which self-similar solutions naturally live. As we will see, the number of technical problems increases dramatically with respect to the case of (NLS). This is due to the higher dispersion, which makes the algebra rather more complicated, and to the presence of derivatives in the non-linear term.

\section{Main results}

\subsection{Notations and functional setting}

We start with some notations. $\hat u$ represents the Fourier transform of a function $u$ (in its space variable $x$ only, if $u$ is a space time function), and we will often denote $p$ the variable dual to $x$ in the Fourier side. We denote by $\mathcal{G}(t)$  the linear KdV group:
\[ \widehat{\mathcal{G}(t) v }(p) = e^{i t p^3} \hat v(p), \]
for any $v \in \cal S'(\m R)$. Given a (space-time) function $u$, we denote $\tilde u$, the \textit{profile} of $u$, as the function defined by
\begin{align} \label{def:u_tilde}
\tilde{u}(t, p):= \widehat{\mathcal{G}(-t)u(t)}(p) = e^{- i t p^3} \hat u(t,p).
\end{align}
In all the following, $C$ denotes various constants, which can change from one line to the next, but does not depend on the other variables which appear. As usual, we use the conventions $a \lesssim b$ and $a = O(b)$ to abbreviate $a \le C b$.

We will also use the Landau notation $a = o_n(b)$ when $a$ and $b$ are two complex quantities (depending in particular of $n$) such that $a/b \to 0$ as $n \to +\infty$; and \emph{mutatis mutandis} $a = o_\e(b)$ when $a/b \to 0$ as $\e \to 0$.

We use often the japanese bracket $\jap y = \sqrt{1+|y|^2}$, and the (complex valued) Airy-Fock function
\begin{equation} \label{def:Ai}
\Ai(z) := \frac{1}{\pi} \int_0^{+\infty} e^{i p z + i p^3} dp.
\end{equation}

For $v \in \cal S'(\m R)$ such that $\hat v \in L^\infty \cap \dot H^1$, and for $t >0$, we define the norm (depending on $t$) with which we will mostly work:
\begin{gather} \label{def:E(t)}
\| v \|_{\q E(t)} := \| \widehat{\mathcal{G}(-t) v} \|_{L^\infty(\m R)} + t^{-1/6} \| \partial_{p} \widehat{\mathcal{G}(-t) v} \|_{L^2((0,+\infty))}.
\end{gather}
Let us remark that we will only consider real-valued functions $u$, and so $\hat u(t,-p) = \overline{\hat u(t,p)}$. As a consequence, the knowledge of frequencies $p >0$ is enough to completely determine $u(t)$, and in the above definition, the purpose of considering $L^2((0,+\infty))$ is to allow a jump at 0. This is necessary because self-similar solutions with $\alpha \ne 0$ (in \eqref{def:CI}) do exhibit such a jump. Indeed, we recall the main result of \cite{CCV19}.

\begin{thm*}
Given $c, \alpha \in \m R$ small enough, there exists unique $a \in \m R$, $A,B \in \m C$ and a self-similar solution $S(t,x) = t^{-1/3} V(t^{-1/3} x)$, where $V$ satisfies
\begin{align}
\text{for } p \ge 2, \quad e^{-it p^3} \hat V(p) & = A e^{i a \ln |p|} + B \frac{e^{3ia \ln |p| -i \frac{8}{9} p^3}}{p^3} + z(p), \label{eq:ss_hf} \\
\text{for } |p| \le 1, \quad e^{-it p^3} \hat V(p) & = c + \frac{3 i \alpha}{2\pi} \sgn(p) + z(p), \label{eq:ss_lf}
\end{align}
where $z \in W^{1,\infty}(\m R)$, $z(0)=0$ and for any $k < \frac{4}{7} $, $|z(p)| + |p z'(p)| = O(|p|^{-k})$ as $|p| \to +\infty$.
\end{thm*}

Let us emphasize that the $\q E(t)$ norm is scaling invariant, in the following sense:
\[ \forall \lambda >0, \quad \| u_\lambda(t) \|_{\q E(t)} = \| u(\lambda t) \|_{\q E(\lambda t)}. \]
In particular, self-similar solutions have constant $\q E(t)$ norm for $t \in (0,+\infty)$.

If $u$ is a space-time function defined on a time interval $I \subset (0,+\infty)$, we extend the above definition and denote
\[ \| u \|_{\q E(I)} := \sup_{t \in I} \| u(t) \|_{\q E(t)} =  \sup_{t \in I} \left( \| \tilde u(t) \|_{L^\infty(\m R)} +  t^{-1/6} \| \partial_p \tilde u(t) \|_{L^2((0,+\infty))} \right). \]

In the same spirit, we define the functional space 
\[ \q E(1): = \{ u \in \cal S'(\m R) : \| u \|_{\q E(1)} < +\infty \}, \]
and  for $I \subset (0,+\infty)$,
\begin{align} \label{def_E}
\q E (I) = \{ u: I \to \cal S'(\m R): \tilde u \in \q  C(I, \q C_b((0,+\infty))), \partial_p \tilde u \in L^\infty(I,L^2((0,+\infty))) \},
\end{align}
endowed with the norm $\| \cdot \|_{\q E(I)}$.

\medskip

\subsection{Main results}
 
We can now state our results. Our main result is a local well-posedness result in the space $\q E(I)$, for initial data $u_1 \in \q E(1)$ at time $t=1$.

\begin{thm} \label{th1}
Let $u_1 \in \q E(1)$. Then there exist $T>1$ and a solution $u \in \q E([1/T,T])$ to \eqref{mkdv} such that $u(1) = u_1$.

Furthermore, one has forward uniqueness. More precisely, let $0 < t_1<t_2$ and $u$ and $v$ be two solutions to \eqref{mkdv} such that  $u, v \in \q E([t_1,t_2])$. If $u(t_1) = v(t_2)$, then for all $t \in [t_1,t_2]$, $u(t) = v(t)$.
\end{thm}

For small data in $\q E(1)$, the solution is actually defined for large times, and one can describe the asymptotic behavior. This is the content of our second result.

\begin{thm} \label{th2}
There exists $\delta >0$ small enough such that the following holds.

If $\| u_1 \|_{\q E(1)} \le \delta$, the corresponding solution satisfies $u \in \q E([1,+\infty))$. Furthermore, let $S$ be the self-similar solution such that
\[ \hat {S}(1, 0^+) = \hat u_1(0^+) \in \m C. \]
Then $\| u(t) - S(t) \|_{L^\infty} \lesssim \| u_1 \|_{\q E(1)} t^{-5/6^-}$ and there exists a profile $U_\infty \in \q C_b(\m R \setminus \{ 0 \}, \m C)$, with $|U_\infty(0^+)| = \lim_{p \to +\infty} |\hat {S}(1,p)|$ is well-defined, and
\[ \left| \tilde u(t,p) - U_\infty(p)\exp\left(-\frac{i\epsilon}{4\pi}|U_\infty(p)|^2\log t\right) \right| \lesssim  \frac{\delta}{\jap{p^3t}^{\frac{1}{12}}} \| u_1 \|_{\q E(1)}. \]
\end{thm}

As a consequence, one has the asymptotics in the physical space. 

\begin{cor} \label{cor_th2}
We use the notations of Theorem \ref{th2}, and let
\[ y= \begin{cases}
\sqrt{-x/3t},& \text{if } x<0, \\
0,& \text{if } x>0.
\end{cases} \]
One has, for all $t \ge 1$ and $x \in \m R$,
\begin{align}
\left|u(t,x)-\frac{1}{t^{1/3}}\partere\Ai\left(\frac{x}{t^{1/3}}\right)U_\infty\left(y\right)\exp\left(-\frac{i\epsilon}{6}|U_\infty(y)|^2\log t\right)\right| \lesssim  \frac{\delta}{t^{1/3}\jap{x/t^{1/3}}^{3/10}}.
\end{align}
\end{cor}

\subsection{Outline of the proofs, comments and complementary results}

In proving Theorem \ref{th1} and \ref{th2}, we use a framework derived from the work of Hayashi and Naumkin \cite{HN01}, improved so that only critically invariant quantities are involved (see Section 3). In particular, we use very similar multiplier identities and vector field estimates. An important new difficulty though is that to perform such energy-type inequalities, the precise algebraic structure of the problem has to be respected (for example, in integration by parts): it seems that one cannot use a perturbative argument like a fixed point, as the method truly requires nonlinear solutions. On the other hand, the rigorous derivation of such inequalities at our level of regularity is quite nontrivial.

This problem does not appear in \cite{HN01} as the authors work in a (weighted) subspace of $H^1$, for which a nice local (and global) well-posedness result hold (\eqref{mkdv} is actually well-posed in $H^s$ for $s \ge 1/4$, see \cite{KPV93}). However, no nontrivial self-similar solution belongs to these spaces, as it can be seen from the lack of decay for large $p$ in \eqref{eq:ss_hf}. Let us also mention the work by Grünrock and Vega \cite{GV09}, where local well-posedness is proved in
\[ \widehat{H}^s_r = \{ u \in \cal S'(\m R) : \| \jap{p}^s \hat u \|_{L^{r'}} < +\infty \} \quad \text{for } 1 < r \le 2, \ s \ge \frac 1 2 - \frac{1}{2r}. \]
This framework is not suitable for our purpose: self-similar belong to $\widehat{H}^0_1$ but not better. 
When finding a remedy for this, let us emphasize again that, due to the jump at frequency 0 for self-similar solutions displayed in \eqref{eq:ss_lf}, one must take extra care on the choice of the functional setting. In particular, smooth functions \emph{are not dense} in $\q E$ spaces (and they can not approximate self-similar solutions). 

In a nutshell, we face antagonist problems coming from low and high frequencies, and we were fortunate enough to manage to take care of both simultaneously.

\bigskip

An important effort of this paper is to solve first an amenable approximate problem (in Section 4), for which we will then derive uniform estimates in the ideology of \cite{HN01}. This approximate problem is actually a variant of the Friedrichs scheme where we filter out high frequencies via a cut-off function $\chi_n$ (in Fourier space). We solve it via a fixed point argument: the cut-off takes care of the lack of decay for large frequencies, but again, smooth functions are not dense in the space $X_n$ where the fixed point is found ($X_n$ is a version of $\q E$ where high frequencies are tamed, but the jump at frequency 0 remains). 

In order to obtain uniform estimates, due to the absence of decay for large frequencies of self-similar solutions, boundary terms cannot be neglected -- unless the cut-off function $\chi_n$ is chosen in a very particular way. 

 At this point, we pass to the limit in $n$ (Section 5), and a delicate but standard compactness argument allows to prove the existence part of Theorem \ref{th1} and Theorem \ref{th2}. The description for large time (the second part of Theorem \ref{th2} and Corollary \ref{cor_th2}) is then a byproduct of the above analysis.

\bigskip

The forward uniqueness result given in Theorem \ref{th1} requires a different argument. We consider the variation of localized $L^2$ norm of the difference $w$ of two solutions. Our solutions do not belong to $L^2$, but we make use of an improved decay of functions in $\q E(I)$ on the right (for $x >0$): in this region, one has a decay of $\jap{x}^{-3/4}$ and therefore they belong to $L^2([0,+\infty)$. The use of a cut-off $\varphi$ which is zero for $x \ll -1$ allows to make sense of the $L^2$ quantity.

When computing the derivative of this quantity, one bad term can not be controlled a priori. Fortunately, if $\varphi$ is furthermore chosen to be non decreasing, this bad term has a sign, and can be discarded as long as one works forward in time (which explains the one-sided result). This is related to a monotonicity property first observed and used by Kato \cite{Kato83}, and a key feature in the study of the dynamics of solitons by Martel and Merle \cite{MM00}. We can then conclude the uniqueness property via a Gronwall-type argument.

Using the forward uniqueness properties, we can improve the continuity properties of the solution $u$: the derivative of its Fourier transform is continuous to the right in $L^2$, see Proposition \ref{prop:cont_u} for the details.

\bigskip

Backward uniqueness for solutions in $\q E$ remains an open problem. One can recover it under some extra decay information, namely that $u_1 \in L^2(\m R)$ (of course this is no longer a critical space). This is the content of our next result, proved in Section \ref{sec:6}.

\begin{prop} \label{prop:L2}
Let $u_1 \in \q E(1) \cap L^2(\m R)$. Then the solution $u \in \q E([1/T,T])$ to \eqref{mkdv} given by Theorem \ref{th1} is unique and furthermore, there is persistence of regularity: $u \in \q C([1/T,T],L^2(\m R))$.
\end{prop}

The stability of self-similar solutions at blow-up time $t=0$, or more generally the behavior of solutions with initial data in $\q E(1)$ near $t=0$ is a challenging question. In this direction, let us present two results which follow from the  tools developed for Theorem \ref{th1}.

The first one, which we prove in Section \ref{sec:7}, is that we can construct solutions to \eqref{mkdv} with a prescribed self-similar profile as $t \to 0^+$.

\begin{prop}[Blow-up solutions with a given profile]\label{prop:blowupgivenprofile}
For $\delta$ sufficiently small, given $g_0\in \cal S'(\m R)$ with $\|g_0\|_{\q E(1)} < \delta$, there exists a solution $u \in \q E((0,+\infty))$ of \eqref{mkdv} such that
\begin{equation}\label{eq:taxablowup}
\forall t >0, \quad \left\|t^{1/3}u(t,t^{1/3}x) -\widehat{g_0}(x)\right\|_{H^{-1}(\m R)} \lesssim \delta t^{1/3}.
\end{equation}
\end{prop}

The second result is concerned with the stability of self-similar blow up. Even the description of the effects of small and smooth perturbations of self-similar solutions \emph{for small time} is not trivial. For example, consider the toy problem of the linearized equation 
\[ \partial_t v + \partial_{xxx} v + \epsilon\partial_x (K^2 v) =0. \]
near the fundamental solution $K(t,x) = t^{-1/3} \Ai(t^{-1/3} x)$ of the linear Korteweg-de Vries equation 
(which is, in some sense, the self-similar solution to the linear problem). The most natural move is to use the estimates of Kenig, Ponce and Vega \cite{KPV00}, which allows to recover the loss of a derivative:
\[ \| v \|_{L^\infty_t L^2_x} \lesssim \| v(0) \|_{L^2_x} + \| K^2 v \|_{L^1_x L^2_t}. \]
Now one can essentially only use Hölder estimate:
\[ \| K^2 v \|_{L^1_x L^2_t} \le \| K \|_{L^4_x L^\infty_t}^2 \| v \|_{L^2_{x,t}}, \]
but, due to the slow decay for $x \ll -1$, $K (t) \notin L^4_x$ for any $t$, and the argument can not be closed. 

We can however prove a stability result of self-similar solutions up to blow-up time, for low frequency perturbations. Given $\alpha>0$ and a sequence $(a_k)_{k\in\m N_0} \subset \m R^+$ satisfying
\[ a_0, a_1=1,\quad \text{and for all } k \ge 0, \quad a_{k}\le \alpha a_{2k+1}, \]
let us define the remainder space
\begin{equation} \label{def:R_a}
\q R_\alpha=\left\{ w \in \q C^\infty(\m R): \sup_{k\ge 0} a_k\|\partial_x^k w\|_{L^2}^2 < \infty \right\}
\end{equation}
endowed with the norm
\begin{equation}
\| w \|_{\q R_\alpha} =  \left(  \sup_{k\ge 0} a_k\|\partial_x^k w \|_{L^2}^2 \right)^{1/2}.
\end{equation}

\begin{prop}[Stability of the self-similar blow-up under $\q R_\alpha$-perturbations] \label{prop:stab}
There exists $\delta>0$ sufficiently small such that, if $w_1\in \q R_\alpha$ and $S$ is a self-similar solution with
	\[ \|w_1\|_{\q R_\alpha}^2 + \alpha\|S(1)\|_{\q E(1)}^2 <\delta, \]
	then the solution $u$ of \eqref{mkdv} with initial data $u_1=S(1)+w_1$ is defined on $(0,1]$ and
	\[ \sup_{t\in (0,1)} \|u(t)-S(t)\|_{\q R_\alpha}^2 < 2\delta. \]
\end{prop}

Obviously, we shrank considerably the critical space by taking smooth perturbations of self-similar solutions, but the above result still shows some kind of stability of self-similar blow up; observe in particular that the blow up time is \emph{not} affected by the perturbation. The study of $\q R_\alpha$ perturbations is done in Section \ref{sec:9}.

\section{Preliminary estimates} \label{sec:3}

Throughout this section, $I \subset (0,+\infty)$ is an interval.

\begin{lem}[Decay estimates]\label{lema:decayE}
Let $u \in \q  C(I,\cal S')$ such that $\| u \|_{\q E(I)} <+\infty$. Then there hold $u \in \q  C(I, L^\infty_{\loc}(\m R))$ and more precisely, for $t \in I$ and $x \in \m R$, one has
\begin{align}\label{eq:est1lem1}
|u(t,x)| & \lesssim \frac{1}{t^{1/3}\jap{|x|/t^{1/3}}^{1/4}} \| u(t) \|_{\q E(t)} \\
\label{eq:est2lem1}
|\partial_x u(t,x) &|\lesssim \frac{1}{t^{2/3}} \jap{|x|/t^{1/3}}^{1/4} \| u(t) \|_{\q E(t)}.
\end{align}
Consequently,
\begin{align} \label{eq:est4lem1}
\|u(t)\|_{L^6}^3 &\lesssim t^{-\frac{5}{6}} \| u(t) \|_{\q E(t)}^3 \\
\label{eq:est5lem1}
\|u(t) \partial_x u(t)\|_{L^\infty} &\lesssim t^{-1} \| u(t) \|_{\q E(t)}^2.
\end{align}
	
Moreover, for $x>t^{1/3}$,
\begin{align} \label{eq:est3lem1}
|u(t,x)| & \lesssim \frac{1}{t^{1/3}\jap{x/t^{1/3}}^{3/4}} \| u(t) \|_{\q E(t)}, \\
|\partial_x u(t,x)| & \lesssim \frac{1}{t^{2/3}\jap{x/t^{1/3}}^{1/4}} \| u(t) \|_{\q E(t)},
\end{align}
and for $x<-t^{1/3}$,
\begin{equation}\label{eq:est6lem1}
\left| u(t,x)- \frac{1}{t^{1/3}}\partere \Ai \left(\frac{x}{t^{1/3}}\right) \tilde{u} \left( t,\sqrt{\frac{|x|}{3t}} \right) \right|\lesssim \frac{1}{t^{1/3}\jap{|x|/t^{1/3}}^{3/10}} \| u(t) \|_{\q E(t)}.
\end{equation}
\end{lem}

\begin{proof} 
The statement and proof are very similar to Lemma 2.1 in \cite{HN01}; notice, however, that the norm $\|\cdot\|_{\textbf{X}}$ therein is stronger than ours, so that we in fact need to systematically improve their bounds. For the convenience of the reader, we provide a complete proof. 

We recall that $\tilde u(t)$ is not continuous at $0$, and may (and will) have a jump (because we only control $\| \partial_p u \|_{L^2(0,+\infty)}$: in the following computations $\tilde u(t,0)$ will mean the limit $\tilde u(t,0^+)$). Setting
\[	z = \frac{x}{\sqrt[3]{t}},\quad y=\sqrt{-\frac{x}{t}}\text{ for }x\le 0,\quad y=0\text{ for }x>0, \]
we have the identity
\begin{align}
u(t,x) & =\frac{1}{\pi}\partere \int_0^\infty e^{ipx+ip^3t} \tilde u (t,p)dp, \quad q = p \sqrt[3]t \nonumber \\
& = \frac{1}{\pi\sqrt[3]{t}}\partere \int_0^\infty e^{iqz + iq^3}\left( \tilde u(t,y)+\left( \tilde u \left(t,\frac{q}{\sqrt[3]{t}}\right)-\tilde u (t,y)\right)\right) dq \nonumber  \\
& = \frac{1}{\sqrt[3]{t}}\partere \Ai \left(\frac{x}{\sqrt[3]{t}}\right) \tilde u(t,y) + R(t,x). \label{eq:u_R}
\end{align}
In the case $x\ge 0$, we integrate by parts in the remainder $R$:
\begin{align*}
R(t,x) & = \frac{1}{\pi\sqrt[3]{t}} \partere \int_0^\infty \partial_q\left(qe^{iqz+iq^3}\right)\frac{1}{1+iq(3q^2+z)}\left(\tilde u \left(t,\frac{q}{\sqrt[3]{t}}\right)-\tilde u(t,y)\right) dq \\
& = \frac{1}{\pi\sqrt[3]{t}} \partere \int_0^\infty \frac{e^{iq z  + iq^3}}{1+iq(3q^2+ z )} \\
& \qquad \times \Bigg(\frac{iq(6q^2+ z )}{1+iq(3q^2+ z )}\left(\tilde u\left(t,\frac{q}{\sqrt[3]{t}}\right)-\tilde u(t,y)\right)- \frac{q}{\sqrt[3]{t}} \partial_p \tilde u \left(t,\frac{q}{\sqrt[3]{t}}\right)\Bigg)dq.
\end{align*}
Since
\[	|\tilde u(t,p)-\tilde u(t,0)|\le \left|\int_0^p \partial_p \tilde u(t,q)dq\right| \le \sqrt{p}\|\partial_p \tilde u(t)\|_{L^2((0,+\infty))}, \]
we can estimate the remainder in the following way:
\begin{align*}
|R(t,x)| & \lesssim \frac{1}{\sqrt[3]{t}} \int_0^\infty \frac{1}{1+q(3q^2+ z )} \left( \left| \tilde u \left( t,\frac{q}{\sqrt[3]{t}}\right) - \tilde u(t,0) \right| + \frac{q}{\sqrt[3]{t}} \left| \partial_p \tilde u \left( t,\frac{q}{\sqrt[3]{t}} \right) \right| \right) dq \\
& \lesssim  \frac{1}{\sqrt{t}} \| \partial_p \tilde u(t) \|_{L^2} \int_0^\infty\frac{\sqrt{q}dq}{1+q(3q^2+ z )} \\
& \qquad + \frac{1}{\sqrt[3]{t^2}}\left(\int_0^\infty \left| \partial_p \tilde u\left(t,\frac{q}{\sqrt[3]{t}}\right)\right|^2 dq\right)^{\frac{1}{2}} \left(\int_0^\infty \frac{q^2 dq}{(1+q(3q^2+ z ))^2}\right)^{\frac{1}{2}} \\
&\lesssim \frac{1}{\sqrt{t}} \left(1+\frac{|x|}{t^{1/3}}\right)^{-1/4} \|\partial_p \tilde u(t)\|_{L^2((0,+\infty))}.
\end{align*}
In the case $x < 0$, we denote $r=\sqrt{- z /3}$. Integrating by parts, we get
\begin{align*}
R(t,x) & = \frac{1}{\pi\sqrt[3]{t}} \partere \int_0^\infty \partial_q\left((q-r)e^{iq z +iq^3}\right)\frac{1}{1+3i(q-r)^2(q+r)} \left( \tilde u\left(t,\frac{q}{\sqrt[3]{t}}\right) - \tilde u(t,y)\right) dq \\
& = -\frac{1}{\pi\sqrt[3]{t}}\partere \int_0^\infty \frac{e^{iq z +iq^3}}{1+3i(q-r)^2(q+r)} \Bigg( \frac{3i(q-r)^2(3q+r)}{1+3i(q-r)^2(q+r)}\left(\tilde u \left(t,\frac{q}{\sqrt[3]{t}}\right)-\tilde u(t,y)\right) \\
& \qquad \qquad + \frac{q-r}{\sqrt[3]{t}}\partial_p \tilde u\left(t,\frac{q}{\sqrt[3]{t}}\right)\Bigg)dq - \frac{r}{\pi\sqrt[3]{t}}\partere\frac{\tilde u(t,0)-\tilde u(t,y)}{1+3ir^3}.
\end{align*}
Then we can estimate
\begin{align*}
|R(t,x)| & \lesssim \frac{1}{\sqrt[3]{t}}\int_0^\infty \frac{1}{1+(q-r)^2(q+r)}\left(\left| \tilde u \left(t,\frac{q}{\sqrt[3]{t}}\right)-\tilde u(t,y) \right| + \frac{|q-r|}{\sqrt[3]{t}}\left|\partial_p \tilde u\left(t,\frac{q}{\sqrt[3]{t}}\right)\right|\right)dq \\
& \qquad\qquad + \frac{1}{\sqrt{t}\jap{r}}\|\partial_p \tilde u(t)\|_{L^2((0,+\infty))} \\ 
&\lesssim \frac{1}{\sqrt{t}}\left( \int_0^\infty \frac{\sqrt{|q-r|}dq}{1+(q-r)^2(q+r)} + \left(\int_0^\infty \frac{(q-r)^2dq}{(1+(q-r)^2(q+r))^2} \right)^{1/2}\right)  \\
& \qquad \qquad \times \| \partial_p \tilde u(t)\|_{L^2((0,+\infty))} +  \frac{1}{\sqrt{t}\jap{r}}\|\partial_p \tilde u(t)\|_{L^2((0,+\infty))} \\
& \lesssim \frac{1}{\sqrt{t}\sqrt{\jap{r}}} \|\partial_p \tilde u(t)\|_{L^2((0,+\infty))}.
\end{align*}
It now follows from the decay of the Airy-Fock function $|\text{Ai}( z )|\lesssim \jap{ z }^{-\frac{1}{4}}$ that
\begin{align*}
|u(t,x)| & \lesssim t^{-1/3}  \jap{\frac{x}{t^{1/3}}}^{-1/4} \left(\|\tilde u(t)\|_{L^\infty} + t^{-\frac{1}{6}}\|\partial_p \tilde u(t)\|_{L^2((0,+\infty))} \right) \\
&  \lesssim  \frac{1}{\sqrt[3]{t} \jap{x/t^{1/3}}^{1/4}} \| u(t) \|_{\q E(t)}. 
\end{align*}

This concludes the proof of \eqref{eq:est1lem1}. For \eqref{eq:est2lem1}, we split once again between the cases $x\ge 0$ and $x < 0$. In the second case, we have as in \eqref{eq:u_R}
\begin{gather*}
\partial_x u(t,x) = \frac{1}{\sqrt[3]{t^2}} \Re \Ai' \left( \frac{x}{\sqrt[3]t} \right) \tilde u(t,y) + \tilde R(t,x), \quad \text{with} \\
\tilde R(t,x) :=  \frac{1}{\pi\sqrt[3]{t^2}}\partere \int_0^\infty iq e^{iq z  + iq^3} \left( \tilde u \left(t,\frac{q}{\sqrt[3]{t}}\right)-\tilde u (t,y)\right)  dq .
\end{gather*}
Analogous computations done for $R$ yield
\begin{align*}
|\tilde R(t,x)|&\lesssim \frac{1}{\sqrt[3]{t^2}}\int_0^\infty \frac{q}{1+(q-r)^2(q+r)}\left(\left| \tilde u \left(t,\frac{q}{\sqrt[3]{t}}\right)-\tilde u(t,y) \right| + \frac{|q-r|}{\sqrt[3]{t}}\left|\partial_p \tilde u\left(t,\frac{q}{\sqrt[3]{t}}\right)\right|\right) dq \\
& \lesssim \frac{\|\tilde u(t)\|_{L^\infty}}{\sqrt[3]{t^2}}\int_0^\infty \frac{qdq}{1+3(q-r)^2(q+r)} \\
& \qquad  + \frac{\|\partial_p \tilde u(t)\|_{L^2((0,+\infty))}}{\sqrt[6]{t^5}} \left(\int_0^\infty \frac{q^2(q-r)^2dq}{(1+3(q-r)^2(q-r))^2}\right)^{\frac{1}{2}} \\
&\lesssim \frac{1}{t^{2/3}}\left(1+\frac{|x|}{t^{1/3}}\right)^{1/4}\left(\|\tilde u(t)\|_{L^\infty} + t^{-\frac{1}{6}}\|\partial_p \tilde u(t)\|_{L^2((0,+\infty))} \right) \\
& \lesssim \frac{1}{t^{2/3}} \jap{\frac{x}{t^{1/3}}}^{1/4} \|u(t) \|_{\q E(t)},
\end{align*}
and the bound for $\partial_x u$ follows from the bound on the Airy-Fock function $|\Ai '( z )| \lesssim \jap{ z }^{\frac{1}{4}}$. 

For $x\ge 0$, we write
\begin{align*}
\partial_x u(t,x) & = \frac{1}{\pi\sqrt[3]{t^2}}\partere\int_0^\infty ie^{iq z  + iq^3}q\tilde u \left(t,\frac{q}{\sqrt[3]{t}}\right)dq \\
& = \frac{1}{\pi\sqrt[3]{t^2}} \partere\int_0^\infty ie^{iq z  + iq^3} q \left(\tilde u(t,0) + \int_0^q \partial_p \tilde u\left(t,\frac{r}{\sqrt[3]{t}}\right)\frac{dr}{\sqrt[3]{t}} \right) dq\\
& = \frac{1}{\pi t}\partere\int_0^\infty \left(\int_r^\infty ie^{iq z  + iq^3}q dq\right)\partial_p \tilde u\left(t,\frac{r}{\sqrt[3]{t}}\right)dr.
	\end{align*}
	Applying Cauchy-Schwarz, and as for $z,r >0$, we have
\[ \int_r^\infty e^{iq z  + iq^3}q dq \lesssim \frac{1}{z+r^2}, \]
we obtain
\begin{align*}
|\partial_x u(t,x)| & \lesssim \frac{1}{\sqrt[6]{t^5}}\|\partial_p \tilde u(t)\|_{L^2((0,+\infty))} \left\|\int_r^\infty e^{iq z  + iq^3}q dq \right\|_{L^2((0,\infty),dr)} \\
& \lesssim \frac{1}{\sqrt[3]{t^2}}\jap{ z }^{-\frac{1}{4}} \|\partial_p \tilde u(t)\|_{L^2((0,+\infty))}.
\end{align*}
Hence \eqref{eq:est2lem1} follows. The estimate for $\partial_x u$ in \eqref{eq:est3lem1} is also a consequence of the above estimate. 

Now we prove the first estimate in \eqref{eq:est3lem1}. To that end, we integrate by parts the expression for $u$:
	\begin{align*}
	u(t,x) & = \frac{1}{\pi}\int_0^\infty e^{ipx+ip^3t}\tilde u(t,p)dp = \frac{\tilde u(t,0)}{\pi x}-\frac{1}{\pi}\int_0^\infty e^{ipx+ip^3t}\partial_p\left(\frac{\tilde u(t,p)}{x+3p^2t}\right)dp \\ & = \frac{\tilde u(t,0)}{\pi x}-\frac{1}{\pi}\int_0^\infty e^{ipx+ip^3t}\frac{\partial_p \tilde u(t,p)}{x+3p^2t}dp + \int_0^\infty e^{ipx+ip^3t}\frac{6pt \tilde u(t,p)}{(x+3p^2t)^2}dp.
	\end{align*}
	The first and third terms are bounded directly, while the second term is bounded using Cauchy-Schwarz:
\begin{align*}
\MoveEqLeft \left|\int_0^\infty e^{ipx+ip^3t}\frac{\partial_p \tilde u(t,p)}{x+3p^2t}dp \right| \lesssim \|\partial_p \tilde u(t)\|_{L^2((0,+\infty))}\left(\int_0^\infty \frac{dp}{(x+3p^2t)^2}  \right)^{1/2}\\
&\lesssim t^{1/6}\frac{1}{t^{1/4}x^{3/4}} \|\partial_p \tilde u(t)\|_{L^2((0,+\infty))} \lesssim \frac{1}{t^{1/3}\jap{x/t^{1/3}}} \|\partial_p \tilde u(t)\|_{L^2((0,+\infty))}.
\end{align*}
Finally, estimate \eqref{eq:est6lem1} follows from \cite[Lemma 2.9]{GPR16}. For completeness, we present the proof:  define $\ell_0\in \mathbb{Z}$ so that 
\[ 2^{\ell_0}\sim t^{-1/3}(|x|/t^{1/3})^{-1/5}. \]
We split the estimate for 
\[ R(t,x)= \frac{1}{\pi}\text{Re }\int_0^{\infty} e^{it\Phi(p)}(\tilde u(t,p)-\tilde u(t,y))dp, \quad \Phi(p)=\frac{x}{t}p + p^3, \]
in three regions, using appropriate cut-off functions $\chi_A+\chi_B+\chi_C=1$:

\medskip

\emph{Region A: $|p-y|\ge y/2$}. Over this region, $\partial_p\Phi(p)\gtrsim \max\{ y, p\}$. Then an integration by parts yields
\begin{align*}
&\left| \int_0^\infty e^{it\Phi(p)}(\tilde u(t,p)-\tilde u(t,y))\chi_A(p) dp\right| 
\lesssim \frac{1}{t^{1/3}(|x|/t^{1/3})^{3/4}} \| u(t) \|_{\q E(t)}.
\end{align*}

\medskip

\emph{Region B: $|p-y|\ge 2^{\ell_0}$}. If $|p-y|\sim 2^l$, with $l\ge \ell_0$, then $|\partial_p\Phi(p)|\gtrsim 2^ly$ and the same integration by parts gives
\begin{align*}
\left(\text{contribution of }|p-y|\sim 2^l\right) \lesssim \left(\frac{1}{t^{5/6}2^{l/2}y} + \frac{1}{t2^{l}y}\right) \| u(t) \|_{\q E(t)}.
\end{align*}
Summing in $l\ge \ell_0$,
\begin{align*}
\left|\int_0^\infty e^{it\Phi(p)}(\tilde u(t,p)-\tilde u(t,y))\chi_B(p) dp\right| 
& \lesssim \left(\frac{1}{t^{5/6}2^{\ell_0/2}y} + \frac{1}{t2^{\ell_0}y}\right) \| u(t) \|_{\q E(t)} \\
&  \lesssim\frac{1}{t^{1/3}(|x|/t^{1/3})^{3/10}} \| u(t) \|_{\q E(t)}.
\end{align*}

\medskip

\emph{Region C: $|p-y|\le 2^{\ell_0}$}. We decompose the integral as
\begin{align*}
\MoveEqLeft \int_0^{\infty} e^{it\Phi(p)}(\tilde u(t,p)-\tilde u(t,y))\chi_C(p)dp \\
& =  \int_0^\infty e^{it\Phi(p)-it(\Phi(y) + 3y(p-y)^2)}(\tilde u(t,p)-\tilde u(t,y))\chi_C(p)dp \\
& \quad + e^{it\Phi(y)}\int_0^\infty e^{3ity(p-y)^2}(\tilde u(t,p)-\tilde u(t,y))\chi_C(p)dp \\
& = I_1+ I_2.
\end{align*}
Since $|\tilde u(t,p)-\tilde u(t,y)|\lesssim t^{1/6}|p-y|^{1/2} \| u(t) \|_{\q E(t)}$, one easily bounds these integrals:
\begin{align*}
| I_1 | & \lesssim t2^{4\ell_0}\|u\| \lesssim t^{-1/3}(|x|/t^{1/3})^{4/5} \| u(t) \|_{\q E(t)} \\
| I_2 | &\lesssim t^{1/6}2^{3\ell_0/2}  \| u(t) \|_{\q E(t)} \lesssim t^{-1/3}(|x|/t^{1/3})^{3/10} \| u \|_{\q E(t)}. \qedhere
\end{align*}
\end{proof}

Let $u \in \q  C(I,\cal S')$ be a solution to \eqref{mkdv} in the distributional sense. Taking the Fourier transform of
\[ (\partial_t + \partial_{xxx}) u = -\epsilon\partial_x (u^3) \]
 using $\tilde{u}(t,p)=e^{itp^3}\hat{u}(t,p)$, one obtains
\begin{equation}
\partial_t \tilde{u}(t,p) + \frac{i\epsilon p}{4\pi^2}\iint_{p_1+p_2+p_3=p} e^{it(p^3-p_1^3-p_2^3-p_3^3)}\tilde{u}(t,p_1)\tilde{u}(t,p_2)\tilde{u}(t,p_3)dp_1dp_2 = 0.
\end{equation}
This leads us to define (with the change of variables $p_i = pq_i$)
\begin{equation}
\cal N[u](t,p)= ip^3\iint_{q_1+q_2+q_3=1} e^{-itp^3(1-q_1^3-q_2^3-q_3^3)}\tilde{u}(t,pq_1)\tilde{u}(t,pq_2)\tilde{u}(t,pq_3)dq_1dq_2,
\end{equation}
so that
\begin{align}\label{eq:tildeu}
\partial_t \tilde{u}(t,p)&=-\frac{\epsilon}{4\pi^2}\cal N[u](t,p).
\end{align}

The following result is a stationary phase lemma for $\cal N[u]$. Similar statements may be found in \cite[Lemma 2.4]{HN01} and \cite{GPR16}.

\begin{lem}[Asymptotics of the nonlinearity on the Fourier side]\label{lem:desenvolveoscilatorio}
Let $u \in \q  C(I,\cal S')$ such that $\| u \|_{\q E(I)} <+\infty$. One has the following asymptotic development for $\cal N[u]$: for all $t \in I$ and $p >0$,
\begin{align} \label{eq:N_expansion}
\cal N[u](t,p)&=\frac{\pi p^3}{\jap{p^3t}}\left(i|\tilde{u}(t,p)|^2\tilde{u}(t,p) - \frac{1}{\sqrt{3}}e^{-\frac{8itp^3}{9}}\tilde{u}^3\left(t,\frac{p}{3}\right) \right)+ R[u](t,p)
\end{align}
where the remainder $R$ satisfies the bound
\begin{align} \label{eq:N_remainder}
|R[u](t,p)| \lesssim \frac{p^3\|u(t) \|^3_{\q E(t)}}{(p^3t)^{5/6}\jap{p^3t}^{1/4}}.
\end{align}
\end{lem}

\begin{proof}
This essentially relies on a stationary phase type argument. We must however emphasize that the computations and the estimations of the errors have to be performed very carefully, because our setting allows few integration by parts and functions have limited spatial decay. We postpone the proof to Appendix A.
\end{proof}

\begin{lem}\label{lem:desenvassimpt}
Let $I \subset (0,+\infty)$ be an interval and $t_1 \in I$. Let $u \in \q  C(I,\cal S')$ be a solution to \eqref{mkdv} in the distributional sense such that $\| u \|_{\q E(I)} <+\infty$.

Then, for some universal constant $C$ (independent of $I$), and for all $t \in I$
\begin{align*}
\|\tilde{u}(t)\|_{L^\infty} & \le\|\tilde{u}(t_1)\|_{L^\infty} + C(\|u\|_{\q E(I)}^3 + \|u\|_{\q E(I)}^5), \\
\| \partial_t \tilde u(t) \|_{L^\infty} & \lesssim \frac{1}{t} \| u(t) \|_{\q E(t)}^3.
\end{align*}
Furthermore, if we denote
\begin{align} \label{def:E_u}
E_u(t,p) :=\exp\left(-i\epsilon\int_{t_1}^t\frac{p^3}{ \jap{p^3s}}|\tilde{u}(s,p)|^2ds\right),
\end{align}
one has, for all $t, \tau \in I$
\begin{equation}\label{eq:decayuniformt}
|(\tilde{u}E_u)(t,p)-(\tilde{u}E_u)(\tau,p)| \lesssim (\|u\|_{\q E(I)}^3 + \|u\|_{\q E(I)}^5)\min\left( \frac{|t-\tau|}{\jap{p^3}},\frac{1}{\jap{\tau p^3}^{\frac{1}{12}}} \right).
\end{equation}
\end{lem}

\begin{proof}
Since $\tilde{u}(t,-p)=\overline{\tilde{u}(t,p)}$, it suffices to consider $p>0$. Using Lemma \ref{lem:desenvolveoscilatorio},

\begin{align}
\partial_t \tilde{u} (s,p) & = -\frac{\epsilon}{4\pi}\frac{p^3} {\jap{p^3s}}\left(\frac{1}{\sqrt{3}}e^{-\frac{8isp^3}{9}}\tilde{u}^3(s,p/3) - i|\tilde{u}(s,p)|^2\tilde{u}(s,p)\right)\label{eq:quaseperfil} \\
&+ O\left(\frac{p^3\|u\|_{\q E(t)}^3}{(p^3t)^{5/6}\jap{p^3t}^{1/4}}\right).
\end{align}
Denote $v(t,p)=\tilde{u}(t,p)E_u(t,p)$. Then the integration in time of \eqref{eq:quaseperfil} on $[\tau,t] \subset I$ yields
\begin{align}\label{eq:perfil}
v(t,p) & = v(\tau,p)-\epsilon\int_{\tau}^t\frac{p^3}{4\pi\sqrt{3} \jap{p^3s}}e^{-\frac{8isp^3}{9}}E_u(s,p)\tilde{u}^3\left(s,\frac{p}{3}\right)ds \\
& \qquad+ O\left(p^3\|u\|_{\q E(I)}^3\int_{\tau}^t \frac{ds}{(p^3s)^{5/6}\jap{p^3s}^{1/4}}\right).
\end{align}

We claim that 
\begin{equation}\label{eq:claim}
\left|\int_{\tau}^t\frac{p^3}{ \jap{p^3s}}e^{-\frac{8isp^3}{9}}E_u(s,p)\tilde{u}^3\left(s,\frac{p}{3}\right)ds \right| \lesssim (\|u\|_{\q E(I)}^3+\|u\|_{\q E(I)}^5) \min \left( \frac{1}{\jap{p^3\tau}}, \frac{|t-\tau|}{\jap{p^3}} \right)
\end{equation}
Indeed, integrating by parts,
\begin{align*}
\MoveEqLeft \int_\tau^t\frac{p^3}{ \jap{p^3s}}e^{-\frac{8isp^3}{9}}E_u(s,p)\tilde{u}^3\left(s,\frac{p}{3}\right)ds \\
& = \int_\tau^t \partial_s \left(s e^{-\frac{8isp^3}{9}}\right)\frac{1}{1-\frac{8is p^3}{9}}\frac{p^3}{ \jap{p^3s}}E_u(s,p)\tilde{u}^3\left(s,\frac{p}{3}\right) ds\\ 
& =-\int_{\tau}^{t} E_u(s,p)\frac{e^{-\frac{8isp^3}{9}}}{1-\frac{8is p^3}{9}}\Bigg(\tilde{u}^3\left(s, p/3\right) \frac{ O(p^6 s) }{\left(1-\frac{8isp^3}{9}\right) \jap{p^3s}} \\
& \qquad + \frac{3sp^3}{ \jap{p^3s}}\tilde{u}^2(s,p/3)\tilde{u}_s(s,p/3) + \frac{ip^3sp^3}{\jap{p^3s}^2}\tilde{u}^3(s,p/3)|\tilde{u}(s,p)|^2  \Bigg)ds \\
& \qquad + \left[\frac{E_u(s,p)e^{-\frac{8isp^3}{9}}\tilde{u}^3(s,p/3)}{1-\frac{8isp^3}{9}}\right]_{s=\tau}^{s=t}.
\end{align*}
From \eqref{eq:quaseperfil}, we have 
\[ |\partial_t \tilde{u}(s,p)|\lesssim s^{-1} \| \tilde u (s) \|_{L^\infty}^3 \lesssim s^{-1} \|u\|_{\q E(s)}^3. \]
 Taking absolute values in the above expression,
\begin{align*}
\left| \int_\tau^t\frac{p^3}{ \jap{p^3s}}e^{-\frac{8isp^3}{9}}E_u(s,p)\tilde{u}^3\left(s,\frac{p}{3}\right)ds\right| & \lesssim (\|u\|_{\q E([\tau,t])}^3+\|u\|_{\q E([\tau,t])}^5)\int_{\tau}^{t}\frac{p^3ds}{\jap{p^3s}^{2}} \\
& \lesssim (\|u\|_{\q E(I)}^3+\|u\|_{\q E(I)}^5) \min\left( \frac{1}{\jap{p^3\tau}}, \frac{|t-\tau|}{\jap{p^3}}\right)
\end{align*}
as claimed. We plug this estimate with $\tau = t_1$ in \eqref{eq:perfil},
\begin{align*}
\|\tilde{u}(t)\|_{L^\infty} = \| v(t)\|_{\infty} \le \| v(t_1)\|_{L^\infty} + C(\|u\|_{\q E(I)} ^3 + \|u\|_{\q E(I)}^5).
\end{align*}
Estimate \eqref{eq:decayuniformt} follows from \eqref{eq:claim}:
\begin{align*}
|v(t,p) - v(\tau,p)| & \lesssim \left| \int_\tau^t\frac{p^3}{4\pi\sqrt{3} \jap{p^3s}}e^{-\frac{8isp^3}{9}}E_u(s,p)\tilde{u}^3\left(s,\frac{p}{3}\right)ds\right| \\
& \qquad + O\left(p^3\|u\|_{\q E(I)}^3\int_\tau^t \frac{ds}{(p^3s)^{5/6}\jap{p^3s}^{1/4}}\right)\\
& \lesssim (\|u\|_{\q E(I)}^3 + \|u\|_{\q E(I)}^5)\min\left(\frac{|t-\tau|}{\jap{p^3}},\frac{1}{\jap{\tau p^3}^{\frac{1}{12}}} \right). \qedhere
\end{align*}

\end{proof}

\section{Construction of an approximating sequence} \label{sec:4}

Let $(\chi_n)_{n \in\mathbb{N}}\subset \mathcal{S}(\m R)$ be a sequence of even decreasing functions such that

\begin{itemize}
\item for all $n\in \m N$, $0<\chi_n \le 1$, $\chi_n^{1/2}\in \mathcal{S}(\m R)$,
\item for all $p \in \m R$,  $\chi_n(p) \to 1$ as $n \to +\infty$.
\item $\ds \sup_{p\in \m R}  |p(\chi_n^{1/2})'(p)| \to 0$ as $n \to +\infty$.
\end{itemize}

The existence of such a sequence is not completely obvious, let us sketch how to construct one.

\begin{claim}
There exists a sequence $(\chi_n)_{n \in\mathbb{N}}$ satisfying the above conditions.
\end{claim}

\begin{proof}
Define the function $\varphi_n$ as follows: $\varphi_n$ is even and
\[ \varphi_n(p) = \begin{cases}
1 & \text{if } |p| \le n \\
1 - \frac{1}{n} \ln (p/n) & \text{if } n \le p \le \alpha_n \\
e^{-p} & \text{if } p \ge \alpha_n.
\end{cases} 
\]
where $\alpha_n>0$ is chosen so that $\varphi_n$ is continuous, that is $\ds 1 - \frac{1}{n} \ln \left( \frac{\alpha_n}{n} \right) = e^{-\alpha_n}$. One can check that $\alpha_n \in [n e^n-1, n e^n]$.

It follows that $0 < \varphi_n \le 1$, $\varphi_n$ is non increasing on $[0,+\infty)$, and $\sup_{p \in \m R}  |p \phi_n' (p)| = O(1/n)$. Then let $\psi \in \q D(\m R)$ be non negative, even and $\| \psi \|_{L^1} =1$. One can see that $\chi_N := (\varphi_n * \psi)^2$ answers the question.
\end{proof}

Define, for any $u\in \mathcal{S}'(\m R)$,
\[
\widehat{\Pi_n u}(p) = \chi_n(p)\hat{u}(p).
\]
Throughout this section, we shall study the properties of the solutions of
\begin{equation}\tag{$\Pi_n$-mKdV}\label{pimkdv}
\begin{cases}
\partial_t u + \partial_{xxx} u  + \epsilon\Pi_n \partial_x (u^3) =0,  \\
u(1)=\Pi_n u_1, 
\end{cases}
\end{equation}
where $u_1 \in \q E(1)$ is given. Equivalently, we consider the equation
\begin{equation}
\partial_t \tilde{u} =-\frac{\epsilon\chi_n}{4\pi^2}\cal N[u],\quad \tilde{u}(1)=\chi_n \tilde{u}_1.
\end{equation}
(with the slight abuse of notation $\tilde u_1 = \widehat{ \mathcal{G}(-1) u_1}$). Define
\begin{equation}
\|u\|_{X_n(t)} :=\|  \widehat{\mathcal{G}(-t) u} \chi_n^{-1}\|_{L^\infty} + \left\| \partial_p ( \widehat{\mathcal{G}(-t) u}) \chi_n^{-1/2} \right\|_{L^2((0,+\infty))}
\end{equation}
and the space
\[ X_n(t) := \left\{ u\in \mathcal{S}'(\m R): \| u \|_{X_n(t)} < \infty \right\}.  \]
Similarly, if $I \subset (0,+\infty)$ is an interval and $u$ a space-time function, we denote
\[ \| u \|_{X_n(I)} := \sup_{t \in I} \| u(t) \|_{X_n(t)} = \sup_{t \in I} \| \tilde u(t) \chi_n^{-1}\|_{L^\infty}  + \| \partial_p \tilde u(t) \chi_n^{-1/2} \|_{L^2((0,+\infty))}, \]
and 
\[  X_n(I) := \left\{ u\in \q  C(I,\mathcal{S}'(\m R)): \tilde u \chi_n^{-1} \in \q  C(I, \q C_b((0,+\infty))),\  \partial_p \tilde u \chi_n^{-1/2} \in \q  C(I, L^2((0,+\infty)))  \right\}. \]
Observe that if $u \in \q E(1)$, then 
\[ \|\Pi_n u_1\|_{X_n(1)} \le \|u_1\|_{\q E(1)}. \]

\begin{prop} \label{prop:u_n}
	Given any $u_1\in \q E(1)$, there exists $T_{-,n} <1$, $T_{+,n}>1$ and a unique 
	$u_n \in X_n((T_{-,n},T_{+,n})) $ maximal solution of \eqref{pimkdv}. Moreover, if $T_{+,n}<\infty$, then
\[	\lim_{t\to T_+} \|u_n(t)\|_{X_n(t)} = +\infty. \]
(A similar statement holds at $T_-$).

In particular, $u \in \q E((T_{-,n},T_{+,n}))$.
\end{prop}

\begin{proof}
	This is a standard fixed-point argument (in the estimates below, the implicit constants are allowed to depend on $n$). We work for times larger than 1, the other case is similar. For $T>1, M>0$, let
	\begin{equation}
	B_{n}(T,M)=\left\{ u \in X_n([1,T]):  \|u\|_{X_n([0,T])} \le M \right\}
	\end{equation}
	endowed with the natural distance
\[ 
	d(u,v)=  \|u-v\|_{X_n([0,T])},
\] 
	and
\[
	(\Psi(u))(t,p)= \chi_n (p) \tilde{u}_1(p) -\frac{\epsilon\chi_n(p)}{4\pi^2}\int_1^t \cal N[u](s,p)ds.
\]
	Using the strong decay on the Fourier side, that is, for any $u \in  X_{n}(T,M)$,
\[
\forall t \in [1,T], \forall p \in \m R, \quad |\tilde{u}(t,p)|\le M\chi_n(p), \]
 one may easily obtain the necessary bounds on $\Psi$. Indeed, we estimate
\begin{align}\label{eq:boundNporchiN}
|\cal N[u|(t,p)|&\lesssim |p| \iint_{q_1+q_2+q_3=p} \chi_n(q_1) \chi_n(q_2) \chi_n(q_3) dq_1dq_2 \|u\|_{X_n}^3 \\
& \lesssim |p| \sup_{|q_3| \ge |p/3|} \chi_{n} (q_3) \| \chi_n \|_{L^1}^2\|u\|_{X_n} ^3 \lesssim \|u\|_{X_n}^3, \nonumber
\end{align}
where we used the fact that at least one of the variables $q_1, q_2$ and $q_3$ has modulus at least $|p/3|$ . Hence for $t \in [1,T]$,
\begin{align*}
\|\hat u(t) \chi_n^{-1}\|_{L^\infty} & = \|\tilde{u}(t) \chi_n^{-1}\|_{L^\infty} \lesssim \|\chi_n \tilde{u}_1 \chi_n^{-1}\|_{L^\infty} + \left\| \int_1^t \cal N[u](s)ds \right\|_{L^\infty} \\
& \lesssim \| \tilde{u}_1 \|_{L^\infty} + (T-1)\sup_{t\in[1,T]} \left\| \cal N[u](t) \right\|_{L^\infty}\\
&\lesssim \|\tilde{u}_1 \|_{L^{\infty}} + (T-1)M^3.
\end{align*}
Similar to estimate \eqref{eq:boundNporchiN}, we have
\begin{align*}
|\partial_p\cal N[u](t,p)|&\lesssim t|p|\iint_{q_1+q_2+q_3=p}(|p|^2+|q_3|^2)\chi_n(q_1) \chi_n(q_2) \chi_n(q_3) dq_1dq_2 \|u\|_{X_n}^3\\&\quad+ |p|\iint_{q_1+q_2+q_3=p} \chi_n(q_1) \chi_n(q_2) |\partial_{p}u(q_3)| dq_1dq_2\|u\|_{X_n}^2\\&\lesssim |p| \sup_{|q_3| \ge |p/3|} (p^2+q_3^2)\chi_{n} (q_3) \| \chi_n \|_{L^1}^2\|u\|_{X_n} ^3 \\&\quad+ |p|\left(\iint_{q_1+q_2+q_3=p}\chi_n(q_1)^2 \chi_n(q_2) \chi_n(q_3) dq_1dq_2 \right)^{1/2}\\&\quad\quad\times\left(\iint \chi_{n}(q_2)|\partial_{p}u(q_3)|^2\chi_{n}^{-1}(q_3)dq_2dq_3\right)^{1/2}\|u\|_{X_n}^2\lesssim \|u\|_{X_n}^3.
\end{align*}

This implies the direct bound
\begin{align*}
\MoveEqLeft \| \partial_p\tilde{u}(t) \chi_n^{-1/2} \|_{L^2((0,+\infty))} \\
& \lesssim \| \chi_n \partial_p\tilde{u}_1 \chi_n^{-1/2}\|_{L^2((0,+\infty))} + (T-1) \| \chi_n^{1/2} \|_{H^1}\sup_{t \in [1,T]} \| \cal N[u](t)\|_{W^{1,\infty}} \\
& \lesssim \| \partial_p\tilde{u}_1 \|_{L^2((0,+\infty))} + (T-1)\|u\|_{X_n}^3.
\end{align*}
Thus
\[ \|(\Psi(u)) \|_{X_n([0,T])} \le C\left(\| u_1 \|_{\q E(1)} + (T-1)M^3\right). \]
	Analogous computations yield 
\[ d(\Psi(u),\Psi(v))\le C (T-1)M^2 d(u,v). \]
(since $\cal N$ is a trilinear operator). Choosing $M$ and $T$ such that 
\[ C \left(\|\tilde{u}_1\|_{X_n} + (T-1)M^3\right) \le M, \]
 and $C (T-1)M \le 1/2$, we see that $\Psi: B_{n}(T,M) \to  B_{n}(T,M) $ is a contraction. The result now follows from Banach's fixed point theorem. 
\end{proof}

To conclude the construction of a solution,  we need the time interval on which the approximating sequence is defined to remain wide independently of $n$. To that end, we need some \emph{a priori} bounds.

\begin{lem}[$L^\infty$ bound for \eqref{pimkdv}]\label{lem:desenvassimptaprox}
	Given $u_1\in \q E(1)$, denote $u$ the corresponding solution of \eqref{pimkdv}, given by Proposition \ref{prop:u_n} and defined on $(T_{-,n},T_{+,n})$. Let $I \subset (T_{-,n},T_{+,n})$. Then\begin{equation}\label{eq:Linftyaprox}
\forall t \in I, \quad \| \tilde{u}_n(t) \chi_n^{-1} \|_{L^\infty} \le \| \tilde{u}_1 \chi_n^{-1} \|_{L^\infty} + C(\|u_n\|_{\q E(I)} ^3 + \|u_n\|_{\q E(I)}^5)
	\end{equation}
	and
\[
	\|\partial_t \tilde{u}_n \chi_n^{-1} \|_{L^\infty}\lesssim \frac{1}{t}\|u_n\|_{\q E(I)} ^3.
\]
	Moreover, if one defines
\[
	E^n(t,p)=\exp\left(- i\epsilon\int_1^t\frac{p^3\chi_n(p)}{4\pi \jap{p^3s}}|\tilde{u}_n(s,p)|^2ds\right).
\]
	then for all $t, \tau \in I$,
\[
	|(\tilde{u}E^n)(t,p) - (\tilde{u}E^n)(\tau,p)|\chi_{n}^{-1}(p)\le C(\|u_n\|_{\q E(I)}^3 + \|u_n\|_{\q E(I)}^5)) \min \left( \frac{|t-\tau|}{\jap{p^3}},\frac{1}{\jap{\tau p^3}^{\frac{1}{12}}} \right).
\]
\end{lem}

\begin{proof}
The proof follows the line of Lemma \ref{lem:desenvassimpt}, we leave the details to the reader.
\end{proof}

Now we look for an \emph{a priori} bound for $\partial_p \tilde{u}_n$. Define the operator
\[
\widehat{\mathcal{I}u}(t,p)= i\partial_p \hat{u}(t,p) - \frac{3it}{p}\partial_t\hat{u}(t,p) = ie^{itp^3}\left(\partial_p\tilde{u} - \frac{3t}{p}\partial_t\tilde{u}\right),
\]
which corresponds to the formal operator
\[
x+3t\int_{-\infty}^x \partial_t dx'.
\]
Using the definition, one may check that, if
\[
\widehat{\Pi_n'u}:= \chi_n'\hat{u}
\]
then
\[
\mathcal{I}\left(\Pi_n u \right)= \Pi_n \mathcal{I}u + i\Pi_n'u.
\]
Moreover, if we let $L = \partial_t + \partial_{xxx}$,
\[
\widehat{L\mathcal{I}u}=\widehat{\mathcal{I}Lu} + \frac{3i}{p}\widehat{Lu},\quad  \mathcal{I}(u^3)_x = 3u^2 \left(\mathcal{I}u\right)_x - 3u^3.
\]

\begin{lem}[$\dot{H}^1$ bound for \eqref{pimkdv}]\label{lem:H1aprox}
Given $u_1\in \q E(1)$, the corresponding solution $u_n$ of \eqref{pimkdv} satisfies
\[
\widehat{\mathcal{I}u_n}\in \q C^1((T_{-,n},T_{+,n}), L^2((0,+\infty), \chi_n^{-1}dp)).
\]
There exists a universal constant $\kappa >0$, such that, for $1<t<T_{+,n}$,
\begin{align}
\forall t \in [1,T_{+,n}) \quad \left(\int_0^\infty |\widehat{\mathcal{I}u_n}(t,p)|^2\chi_n^{-1}dp \right)^{1/2} & \le \left(\int_0^\infty |\widehat{\mathcal{I}u_n}(1,p)| \chi_n^{-1} dp \right)^{1/2}t^{\kappa \| u_n \|_{\q E([1,t])}^2} \nonumber\\
&\qquad + o_n(1) \| u_n \|_{\q E([1,t])}^3  t^{1/6}\label{eq:H1aprox_tmaior}, \\
\forall t \in (T_{-,n},1], \quad \left(\int_0^\infty |\widehat{\mathcal{I}u_n}(t,p)|^2\chi_n^{-1}dp \right)^{1/2} &\le \left(\int_0^\infty |\widehat{\mathcal{I}u_n}(1,p)| \chi_n^{-1} dp \right)^{1/2}t^{-\kappa\| u_n \|_{\q E([t,1])}^2} \nonumber\\&\qquad+ o_n(1) \| u_n \|_{\q E([t,1])}^3  t^{1/6}\label{eq:H1aprox_tmenor}.
\end{align}

\end{lem}

\begin{proof}
Fix $\e>0$. First of all, notice that, by Lemma \ref{lem:desenvassimptaprox}, we have
\[
\widehat{\mathcal{I}u_n} = ie^{-itp^3}\left(\partial_p\tilde{u}_n - \frac{3t}{p}\partial_t\tilde{u}_n \right) \in L^2([\e,+\infty), \chi_n^{-1}dp),
\]
which justifies the finiteness of all the following integrations.
On the other hand,
\begin{align*}
\partial_t\widehat{\mathcal{I}u_n} - ip^3\widehat{\mathcal{I}u_n} & = \widehat{L\mathcal{I}u_n} = \widehat{\mathcal{I}Lu_n} - 3 \epsilon \widehat{\Pi_n (u_n^3)} = -\epsilon\widehat{\mathcal{I}(\Pi_n(u_n^3)_x)} -3 \epsilon  \chi_n\widehat{u_n^3}\\
& = -\epsilon\chi_n\widehat{\mathcal{I}(u_n^3)_x} - i\epsilon\chi_n' \widehat{(u_n^3)_x} -  3 \epsilon \chi_n \widehat{u_n^3} = -\epsilon\left(3\chi_n \widehat{u^2(\mathcal{I}u_n)_x} +ip\chi_n' \widehat{u_n^3}\right).
\end{align*}
Multiplying by $\overline{\widehat{\mathcal{I}u_n}} \chi_n^{-1}$, integrating on $\m R\setminus (-\varepsilon,\varepsilon)$ and taking the real part,
\begin{align*}
\frac{1}{2}\frac{d}{dt}\int_{\m R\setminus (-\varepsilon,\varepsilon)} |\widehat{\mathcal{I}u_n}(p)|^2\chi_n^{-1}(p)dp &=-\epsilon \partere \int_{\m R\setminus (-\varepsilon,\varepsilon)}\int_{p_1+p_2=p} \widehat{u_n^2}(p_1)p_2\widehat{\mathcal{I}u_n}(p_2)\overline{\widehat{\mathcal{I}u_n}}(p)dp_2dp \\
&\ \ \  -\epsilon\partere \int_{\m R\setminus (-\varepsilon,\varepsilon)} p\chi_n'(p)\chi_n^{-1/2}(p)\widehat{u_n^3}(p)\overline{\widehat{\mathcal{I}u_n}}(p)\chi_n^{-1/2}(p)dp \\
& =I_1 + I_2.
\end{align*}
For $I_1$, we split the integral in $p_2$:
\begin{align*}
\left| \int_{\m R\setminus (-\varepsilon,\varepsilon)}\int_{-\varepsilon}^{\varepsilon} \overline{\widehat{\mathcal{I}u_n}}(p) \widehat{\mathcal{I}u_n}(p_2)p_2 \widehat{u_n^2}(p_1)dp_2dp \right| & \lesssim \varepsilon\|\widehat{\mathcal{I}u_n}\|_{L^2(\m R\setminus(-\varepsilon,\varepsilon))}\|\widehat{\mathcal{I}u_n}\|_{L^2((0,+\infty))}\|\widehat{u_n^2}\|_{L^1} \\
& \to 0 \quad \text{as } \e \to 0.
\end{align*}
(Indeed $\|\widehat{u_n^2}\|_{L^1} \lesssim \| \tilde u_n \|_{L^1}^2<+\infty$).
Then observe that
\begin{align*}
\MoveEqLeft \partere \int_{\m R\setminus (-\varepsilon,\varepsilon)} \int_{\m R\setminus (-\varepsilon,\varepsilon)}\widehat{u^2_n}(p_1)p_2\widehat{\mathcal{I}u_n}(p_2)\overline{\widehat{\mathcal{I}u_n}}(p)dp_2dp \\
& = \partere \int_{\m R\setminus (-\varepsilon,\varepsilon)} \int_{\m R\setminus (-\varepsilon,\varepsilon)}\widehat{u^2_n}(p_1)(p-p_1)\widehat{\mathcal{I}u_n}(p_2)\overline{\widehat{\mathcal{I}u_n}}(p)dp_2dp \\
&= -\partere \int_{\m R\setminus (-\varepsilon,\varepsilon)} \int_{\m R\setminus (-\varepsilon,\varepsilon)}p_1\widehat{u^2_n}(p_1)\widehat{\mathcal{I}u_n}(p_2)\overline{\widehat{\mathcal{I}u_n}}(p)dp_2dp \\
&\quad- \partere \int_{\m R\setminus (-\varepsilon,\varepsilon)} \int_{\m R\setminus (-\varepsilon,\varepsilon)}\widehat{u^2_n}(p_1)p_2\widehat{\mathcal{I}u_n}(p_2)\overline{\widehat{\mathcal{I}u_n}}(p)dp_2dp.
\end{align*}
Hence, if we define the operator 
\[ \widehat{\m 1_\e v}=\m 1_{\m R\setminus(-\varepsilon,\varepsilon)}\hat{v}, \]
then
\begin{align*}
\MoveEqLeft \left|\partere \int_{\m R\setminus (-\varepsilon,\varepsilon)} \int_{\m R\setminus (-\varepsilon,\varepsilon)}\widehat{u_n^2}(p_1)p_2\widehat{\mathcal{I}u_n}(p_2)\overline{\widehat{\mathcal{I}u_n}}(p)dp_2dp \right| \\
&= \frac{1}{2}\left|\partere \int_{\m R\setminus (-\varepsilon,\varepsilon)} \int_{\m R\setminus (-\varepsilon,\varepsilon)}p_1\widehat{u_n^2}(p_1)\widehat{\mathcal{I}u}(p_2)\overline{\widehat{\mathcal{I}u_n}}(p)dp_2dp\right| \\
& \lesssim \left|\iint p_1\widehat{u_n^2}(p_1)\widehat{\m 1_\varepsilon \mathcal{I}u}(p_2)\overline{\widehat{\m 1_\varepsilon \mathcal{I}u_n}}(p)dp_2dp\right|\lesssim \left|\int (u_n^2)_x(x)|\m 1_\varepsilon(\mathcal{I}u_n)(x)|^2dx\right| \\
&\lesssim \|u_n \partial_x u_n\|_{L^\infty}\|\m 1_\varepsilon(\mathcal{I}u_n)\|_{L^2}^2 \lesssim \|u_n \partial_x u_n\|_{L^\infty}\int_{\m R\setminus (-\varepsilon,\varepsilon)} |\widehat{\mathcal{I}u_n}(p)|^2\chi_n^{-1}(p)dp.
\end{align*}

Now, to estimate $I_2$, we use Cauchy-Schwarz:
\begin{align*}
|I_2| &\lesssim \left(\int_{\m R\setminus (-\varepsilon,\varepsilon)} |\widehat{\mathcal{I}u_n}(p)|^2\chi_n^{-1}(p)dp\right)^{1/2}\left(\int|p\chi_n'(p)\chi_n^{-1/2}|^2|\widehat{u_n^3}|^2dp \right)^{1/2}\\
&\lesssim  \left(\int_{\m R\setminus (-\varepsilon,\varepsilon)} |\widehat{\mathcal{I}u_n}(p)|^2\chi_N^{-1}(p)dp\right)^{1/2} \| u_n(t) \|_{L^6}^3 \sup_{p\in\m R}|p(\chi_n^{1/2})'(p)|
\end{align*}
Here we crucially use the third condition on $\chi_n$.
Putting together these estimates, using Lemma \ref{lema:decayE} and the symmetry $\hat u_n(t,-p)=\overline{\hat u_n(t,p)}$,
\begin{align}
\MoveEqLeft \left|\frac{d}{dt}\left(\int_\varepsilon^\infty |\widehat{\mathcal{I}u_n}(p)|^2\chi_n^{-1}(p)dp\right)^{1/2}\right|\nonumber \\
&\label{eq:estH1aprox} \lesssim \| u \partial_x u \|_{L^\infty}\left(\int_\varepsilon^\infty |\widehat{\mathcal{I}u_n}(p)|^2\chi_n^{-1}(p)dp\right)^{1/2} + o_\varepsilon(1) + o_n(1) \| u_n (t) \|_{L^6}^3 \\
& \lesssim \frac{\| u(t) \|_{\q E(t)} ^2}{t}\left(\int_\varepsilon^\infty |\widehat{\mathcal{I}u_n}(p)|^2\chi_n^{-1}(p)dp\right)^{1/2} + o_\varepsilon(1) + o_n(1)\frac{\| u_n(t) \|_{\q E(t)}^3}{t^{5/6}}\nonumber
\end{align}
It follows that, for $t\ge 1$ and some universal constant $kappa>0$,
\begin{align*}
\left(\int_\varepsilon^\infty |\widehat{\mathcal{I}u_n}(p)|^2\chi_n^{-1}(t,p)dp\right)^{1/2} & \le \left(\int_\varepsilon^\infty |\widehat{\mathcal{I}u_n}(p)|^2(1,p)\chi_n^{-1}(p) dp \right)^{1/2}t^{\kappa \| u_n \|_{\q E(1,t)}^2} \\
&\quad + o_\varepsilon(t) +  \| u_n(t) \|_{\q E(t)}^3 o_n(t^{1/6}).
\end{align*}
Taking $\varepsilon\to 0$, the result follows. An analogous computation yields the inequality for $t<1$.
\end{proof}

\begin{prop}[Global existence]\label{prop:boundapprox}
	Given $u_1\in \q E(1)$ small, let $u_n$ be the unique maximal solution of \eqref{pimkdv} given by Proposition \ref{prop:u_n}. Then there exists $T=T(\|u_1\|_{\q E(1)})<1$ such that, if $n$ is large enough, $u_n$ is defined on $[T,+\infty)$ and
	\begin{equation}
	\| u_n \|_{\q E([T,+\infty))} \le C \|u_1\|_{\q E(1)}.
	\end{equation}
\end{prop}

\begin{proof}
	Fix $\delta_0>\|u_1\|_{\q E(1)}$. Define
	\begin{equation}
	f_n(t)= \| \tilde{u}_n(t) \chi_n^{-1}\|_{L^\infty} + t^{-1/6}\|\partial_p \tilde{u}_n(t) \chi_N^{-1/2}\|_{L^2((0,+\infty))}
	\end{equation}
	and let $J_n$ be the maximal connected interval containing $t=1$ such that
	\[
 f_n(t) \le 4 C \delta_0,\quad t\in J_n.
	\]
	For $\delta_0$ sufficiently small and some $T<1$ close to 1, it follows from  Lemma \ref{lem:H1aprox} that, given $t\in J_n$, $t>T$,
	\[
	\left(\int_0^\infty |\widehat{\mathcal{I}u}(t,p)|^2\chi_n^{-1}dp \right)^{1/2}\le 2\left(\left(\int_0^\infty |\widehat{\mathcal{I}u}(1,p)|^2 \chi_n^{-1}dp\right)^{1/2} + o_n(1) f_n(s)^3\right)t^{1/6}.
	\]
	Recalling that
	\[
	\partial_p\tilde{u}_n = -ie^{itp^3}\widehat{\mathcal{I}u_n} + \frac{3t}{p}\partial_t\tilde{u}_n =  -ie^{itp^3}\widehat{\mathcal{I}u_n} + 3t\chi_n e^{-itp^3}\widehat{u^3_n},
	\] 
	we derive the bound for $\partial_p\tilde{u}_n$:
	\begin{align*}
	\| \partial_p\tilde{u}_n(t) \chi_n^{-1/2}\|_{L^2((0,+\infty))} & \le \| \widehat{\mathcal{I}u_n}(t) \chi_n^{-1/2} \|_{L^2((0,+\infty))} + 3t\| \widehat{u_n^3}(t) \chi_n^{1/2} \|_{L^2} \\
	& \lesssim t^{1/6}\left(\| \widehat{\mathcal{I}u_n}(1) \chi_n^{-1/2}\|_{L^2((0,+\infty))} + o_n(1)f_n(t)\right) + 3t \|u_n \|_{L^6}^3 \\
	& \lesssim t^{1/6}\left(\| \widehat{\mathcal{I}u_n}(1)\chi_n^{-1/2} \|_{L^2((0,+\infty))} + o_n(1)f_n(t) + f_n (t)^3\right).
	\end{align*}
	Together with the $L^\infty$ bound \eqref{eq:Linftyaprox}, we infer
	\begin{align*}
	f_n(t) & \le C\left(\|\chi_n^{-1}\tilde{u}(1)\|_{L^\infty} + \|\chi_n^{-1/2}\partial_p\tilde{u}(1)\|_{L^2((0,+\infty))} + o_n(1) f_n(t) + f_n(t)^3\right) \\
	& \le C\left(\|u_1\|_{\q E(1)} + o_N(1)f_n(t) + f_n(t)^3\right).
	\end{align*}
	If $n$ large and $4C\|u_1\|<\delta_0$, then a continuity argument implies that
	\begin{equation}
	\forall t \in J_n,\ t>T, \quad f_n(t) \le 2C\|u_1\|_{\q E(1)} < \frac{\delta_0}{2}.
	\end{equation}
	Hence $J_n$ must be equal to $[T,T_{+,n})$. By the definition of $J_n$ and the blow-up alternative, $T_{+,n} = +\infty$.
\end{proof}

\section{Well-posedness on the critical space} \label{sec:5}

\begin{prop}[Existence for small data]\label{prop:exist}
	 There exists $C,\delta>0$ such that, given $u_1\in \q E(1)$ with $\|u_1\|_{\q E(1)}<\delta$, there exist $T=T(\|u_1\|_{\q E(1)})<1$ and a unique $u\in \q{E}([T,\infty))$ solution of \eqref{mkdv} in the distributional sense such that $u(1)=u_1$. Moreover, there exists a universal constant $C>1$ such that
	 \begin{equation}\label{eq:boundfinal}
	 	 \|u\|_{\q{E}([T,\infty))}\le C \| u_1 \|_{\q E(1)}.
	 \end{equation}
\end{prop}

\begin{proof}
	\emph{Step 1. Approximate solutions and a priori bounds.} For each $n\in \mathbb{N}$, define $u_n$ as the unique solution of \eqref{pimkdv}. By Proposition \ref{prop:boundapprox}, for $n$ large enough, $u_n$ is defined on $[T,+\infty)$ and
\[
	\| u_n \|_{\q E([T,+\infty))}\le C\delta.
\]
	
\bigskip
	
	\emph{Step 2. Convergence on the profile space.} Since the sequence $(\tilde{u}_n)_{n\in\mathbb{N}}$ is uniformly bounded in $\q C([T,\infty), \dot{H}^1((0,+\infty))\cap L^\infty(\m R))$,
 the Sobolev embedding implies that
\[
	|\tilde{u}_n(t,p)-\tilde{u}_n(t,q)|\lesssim |p-q|^{1/2},\quad p,q>0.
\]
	Moreover, by Lemma \ref{lem:desenvassimptaprox}, we have $\tilde{u}_n$ uniformly bounded in $  W^{1,\infty}([T,\infty),L^\infty(\m R)).$ Hence $(\tilde{u}_n)_{n\in\mathbb{N}}$ is equicontinuous on $[T,\infty)\times [0,R]$, for any $R>0$. By Ascoli-Arzelà theorem, we conclude that there exists $\tilde{u}\in C_b([T,\infty)\times (0,+\infty))$ such that, up to a subsequence,
\[
	\tilde{u}_N \to \tilde{u}\quad \text{uniformly in }[T,T']\times [0,R], \ T'>T,\ R>0.
\]
	Given $p<0$, we set
\[
	\tilde{u}(t,p)=\overline{\tilde{u}(t,-p)}.
\]
	
	\emph{Step 3. $\tilde{u}\in C_b([T,\infty)\times \m R)\cap L^\infty((T,\infty), \dot{H}^1((0,+\infty)))$}. Define
\begin{align*}
	E_u(t,p) & = \exp\left(-i\epsilon\int_1^t\frac{p^3}{4\pi \jap{p^3s}}|\tilde{u}(s,p)|^2ds\right), \\
	 E^n(t,p) & = \exp\left(-i\epsilon\int_1^t\frac{\chi_n(p)p^3}{4\pi \jap{p^3s}}|\tilde{u}_n(s,p)|^2ds\right).
\end{align*}
	By Lemma \ref{lem:desenvassimptaprox},
\[
	|\tilde{u}_nE^n(t.p)-\tilde{u}_nE^n(s,p)|\lesssim \frac{|t-s|}{\jap{p^3}}.
\]
	Taking $n\to\infty$, we get
\[
	|\tilde{u}E_{u}(t.p)-\tilde{u}E_{u}(s,p)|\lesssim \frac{|t-s|}{\jap{p^3}},
\]
	which means that $\tilde{u}E_{u}\in \q C([T,\infty), L^\infty(\m R))$. On the other hand,
\[
	|E_u(t,p)-E_u(s,p)|\lesssim \int_s^t \frac{p^3}{ \jap{p^3s}}|\tilde{u}(s,p)|^2ds \lesssim |t-s|.
\]
	Hence, when $t\to s$,
\[
	\|\tilde{u}(t)-\tilde{u}(s)\|_{L^\infty(\m R)}\le \|\tilde{u}(t)E_u(t)-\tilde{u}(s)E_u(s)\|_{L^\infty(\m R)} + \|\tilde{u}(s)\left(E_u(t)-E_u(s) \right)\|_{L^\infty(\m R)} \to 0
\]
	and so $\tilde{u}\in \q C([T,\infty), L^\infty(\m R))$.
	
\bigskip
	
	Fix $t\in [T,\infty)$. Since $(\tilde{u}_n(t))_{n\in\mathbb{N}}$ is bounded in $\dot{H}^1((0,+\infty))$, up to a subsequence, there exists $g(t)\in L^2((0,+\infty))$ such that
\[
	\partial_p\tilde{u}_n(t)\rightharpoonup g(t),\quad \|g(t)\|_{L^2((0,+\infty))}\le \liminf \|\partial_p\tilde{u}_n(t)\|_{L^2((0,+\infty))} \lesssim \delta.
\]
	Since $\tilde{u}_n(t)\to \tilde{u}(t)$ in $L^\infty_{loc}(\m R)$, we have $\tilde{u}(t)\in \dot{H}^1((0,+\infty))$ and $\partial_p\tilde{u}(t)=g(t)$. Moreover, the uniform bound on $g(t)$ implies
\[
	\tilde{u}\in C_b([T,\infty) \times \m R)\cap L^\infty((T,\infty), \dot{H}^1((0,+\infty))).
\]

	\emph{Step 4. Convergence on the physical space.}
	We already now that
\[
	|\tilde{u}_n(t,p)-\tilde{u}_n(t,0^+)|, |\tilde{u}(t,p)-\tilde{u}(t,0^+)|\lesssim |p|^{1/2}.
\]
	Hence, $\tilde{u}_n(t)\to \tilde{u}(t)$ in $\mathcal{S}'(\m R)$. Therefore  $u(t)=(e^{itp^3}\tilde{u})^{\vee}(t)$ is well-defined and we have
\[
	u_n\to u \text{ in }\mathcal{D}'((T,\infty)\times \m R),
\]
	
	We now claim that $u$ is in $L^\infty((T,\infty)\times\m R)$ and that $u_n(t)\to u(t)$ in $L^\infty(\m R)$, for any $t\in [T,\infty)$. Since $(u_n)_{n\in\mathbb{N}}$ is
	bounded in $\q E([T,\infty))$, Lemma \ref{lema:decayE} implies that, for any $K\subset \m R $ compact and $T'>T$,
\begin{gather*}
	\forall t \in [T,T'], \quad \|u_n(t)\|_{L^\infty(K)}, \|(u_n)_x(t)\|_{L^\infty(K)} \lesssim_{T',K} \delta, \\
\text{and} \quad \forall x\in \m R, \ \forall t \in [T,T'], \quad	|u_n(t,x)|\lesssim C(T') \jap{x}^{-\frac{1}{4}}.
\end{gather*}
	Again by Ascoli-Arzelà, there exists $h(t)\in \q C(\m R)$ such that
\[
	u_{n_k}(t)\to h(t) \quad \text{uniformly in } K, \ K\subset \m R\text{ compact.}
\]
	and
\[
	|h(t,x)|\lesssim C(T') \jap{x}^{-\frac{1}{4}},\quad x\in\m R.
\]
	This implies that $h$ is, in fact, bounded over $(T,T')\times \m R$. Since $u_n\to u$ in the distribution sense, $h=u$. Hence the limit $h(t)$ is unique and we conclude that the whole sequence $(u_n(t))_{n\in\mathbb{N}}$ must converge to $h(t)$:
\[
	u_n(t)\to h(t)=u(t) \quad \text{uniformly in } K, \ K\subset \m R\text{ compact.}
\]
	Finally, the uniform decay of $u_n$ and $u$ imply that this convergence holds over $\m R$,
\[
	u_n(t)\to u(t) \text{ in }L^\infty(\m R).
\]
	The claim is proven.
	
\bigskip

	\emph{Step 5. $u$ is a solution of \eqref{mkdv} in $\q E([T,+\infty))$}. Since $u_n(t)\to u(t) \text{ in }L^\infty(\m R)$, one has  
\[ (u_n)^3\to u^3 \text{ in } \mathcal{D}'((T, \infty)\times\m R). \]
	Recalling that $(\partial_t + \partial_{xxx})u_n=-\epsilon\Pi_n((u_n)^3)_x$, one may now pass to the limit in the distributional sense and
\[
	(\partial_t + \partial_{xxx})u=-\epsilon(u^3)_x.
\]
	By Step 3, $u\in \q E([T,+\infty))$ and the bound \eqref{eq:boundfinal} follows from the corresponding bound for $u_n$. The proof is complete.
\end{proof}

\begin{nb}\label{rmk:boundI}
	As a consequence of the above proof and  Lemma \ref{lem:H1aprox}, one may easily see that, if $\|u_1\|_{\q E(1)}<\delta<\delta_0$, then
		\begin{align}
\forall t \ge 1, \quad	\|\widehat{\mathcal{I}u}(t)\|_{L^2((0,+\infty))} & \le \|\widehat{\mathcal{I}u}(1)\|_{L^2((0,+\infty))} t^{\kappa \delta^2},\quad \text{and} \\
\label{eq:estIparatras}
\forall t \le 1, \quad	\|\widehat{\mathcal{I}u}(t)\|_{L^2((0,+\infty))} & \le \|\widehat{\mathcal{I}u}(1)\|_{L^2((0,+\infty))} t^{- \kappa\delta^2}.
\end{align}
\end{nb}

We now consider the large data case. Here, the only delicate point is to prove that the lifespan $[T_{-,n},T_{+,n}]$ does not become trivial as $n$ tends to $\infty$. Afterwards, the arguments of the previous proof may be applied \textit{mutatis mutandis}.

\begin{lem}[Uniform local existence for large data]\label{lem:uniformtime}
	Given $u_1\in \q E(1)$, there exists $T_-(u_1)<1$, $T_+(u_1)>1$, $C = C(\| u_1 \|_{\q E(1)}$ such that, for large $n$, the corresponding solution $u_n$ of \eqref{pimkdv} is defined on $[T_-(u_1), T_+(u_1)]$ and
\[
	\|u_n\|_{\q E([T_-(u_1), T_+(u_1)])}\le C \|u_1\|_{\q E(1)}.
\]
\end{lem}

\begin{proof} 
Due to the critical nature of the space $\q E$, we are unable to obtain a uniform time-continuity estimate for the solutions $u_n$. Instead, we argue by contradiction. We focus on showing $T_+(u_1)>1$, the other case being completely analogous.

Let $C_1>0$ be a large constant to be chosen later. For each $n$, let $t_n>1$ be the first time satisfying
\[
	\| u_n(t_n) \|_{\q E(t_n)}= C_1 \|u_1\|_{\q E(1)}.
\]
	Suppose, for the sake of contradiction, that $t_n\to 1$, in particular $t_n \le 2$. The uniform bound of $\tilde{u}_n(t_n)$ in $L^\infty\cap \dot{H}^1((0,+\infty))$ implies the existence of $v$ such that
\[
	\partial_{p}\tilde{u}_n(t_n)\rightharpoonup \partial_{p}v \mbox{ in }L^2(0,+\infty), \quad \tilde{u}_n(t_n)\to v\mbox{ in }L^\infty(K), K\subset \m R\mbox{ compact}
\]
	On the other hand, by Lemma \ref{lem:desenvassimptaprox}, 
\begin{align}
	|(\tilde{u}_nE^n)(t_n,p) - \tilde{u}_1(p)|&=|(\tilde{u}_nE^n)(t_n,p) - (\tilde{u}_nE^n)(1,p)|+|\tilde{u}_n(1,p) - \tilde{u}_1(p)| \nonumber\\
	&\lesssim (\| u_n(t_n) \|_{\q E(t_n)}^3 + \| u_n(t_n) \|_{\q E(t_n)}^5) \|  |t_n-1|+|\Pi_n(p) - 1| \to 0\label{eq:contradLinfty}
\end{align}
which means that $v=\tilde{u}_1$. Moreover, the decay estimates of Lemma \ref{lema:decayE} imply that 
\[
u_n(t_n)\to u_1\text{ in } L^6(\m R).
\]
Due to  \eqref{eq:H1aprox_tmaior} from Lemma \ref{lem:H1aprox}, we have
\begin{align*}
\|(\mathcal{I}u_n)(t_n)\|_{L^2((0,+\infty))}& \lesssim \left(\int_0^\infty |\widehat{\mathcal{I}u_n}(1)|^2\chi_n^{-1}dp\right)^{1/2} t_n^{C_1^2 \kappa \| u_1 \|_{\q E(1)}^2} +  o_n(1) t_n^{1/6}  \| u_n(1) \|_{\q E(1)}.  \\ 
\end{align*}
For $n$ large,  $t_n^{C_1^2 \kappa \| u_1 \|_{\q E(1)}^2} \le 2$. Using once more the formula
\begin{align} \label{eq:Iu_n}
	\partial_p\tilde{u}_n  =  -ie^{itp^3}\widehat{\mathcal{I}u_n} + 3t \chi_n e^{-itp^3}\widehat{u^3_n},
	\end{align}
we get that for large $n$,
\begin{align*} 
\| \mathcal{I}u_n(1) \chi_n^{-1/2} \|_{L^2((0,+\infty))} & \lesssim \| \partial_p\tilde{u}_n(1) \chi_n^{-1/2} \|_{L^2((0,+\infty))} + \| \widehat{ u_n^3(1)} \chi_n^{1/2} \|_{L^2((0,+\infty))} \\
& \lesssim \| u_1 \|_{\q E(1)} + \| u_1 \|_{L^6}^3  \lesssim \| u_1 \|_{\q E(1)} (1 + \| u_1 \|_{\q E(1)}^2).
\end{align*}
As a consequence, we get 
\[ \|(\mathcal{I}u_n)(t_n)\|_{L^2((0,+\infty))} \lesssim \| u_1 \|_{\q E(1)} (1 + \| u_1 \|_{\q E(1)}^2). \]
In the above estimate, we emphasize that the implied constant does not depend on $C_1$.
Using \eqref{eq:Iu_n} yet another time gives 
\begin{align*} 
\| \partial_p\tilde{u}_n (t_n) \|_{L^2((0,+\infty))} & \le \|(\mathcal{I}u_n)(t_n)\|_{L^2((0,+\infty))} + 3t_n \|  \widehat{ u_n^3(1)} \|_{L^2((0,+\infty))} \\
& \lesssim \| u_1 \|_{\q E(1)} (1 + \| u_1 \|_{\q E(1)}^2).
\end{align*}
Moreover, it follows from \eqref{eq:contradLinfty} that $\| \tilde u_n(t_n) \|_{L^\infty} \le 2 \| u_1 \|_{\q E(1)}$ for large $n$.  In other words, for some absolute constant $C_0$ and large $n$, 
\[ \| u_n(t_n) \|_{\q E(t_n)} \le C_0  \| u_1 \|_{\q E(1)} (1 + \| u_1 \|_{\q E(1)}^2). \]
Choose now $C_1 = 2 C_0 (1 + \| u_1 \|_{\q E(1)}^2)$: this is a contradiction. Hence $t_n \not\to 1$, and the proof is complete.
\end{proof}

We can now follow the same arguments as for the proof of Proposition \ref{prop:exist}, and get the analoguous result for large data given below.

\begin{prop}[Existence for large data]\label{prop:exist_u_E}
Let $u_1\in \q E(1)$. There exist $T_{-}(u_1)<1$ and $T_+(u_1) >1$ and a solution $u\in \q{E}([T_-(u_1),T_+(u_1)])$ solution of \eqref{mkdv} in the distributional sense such that $u(1)=u_1$.
\end{prop}

We now turn to the forward uniqueness result. It relies on completely different arguments, related to a monotonicity  formula.

\begin{prop}[Forward uniqueness]\label{prop:uniq}
	If $u,v\in \q E([t_1,t_2])$ are two solutions of \eqref{mkdv} and $u(t_1)=v(t_1)$, then $u\equiv v$.
\end{prop}

\begin{proof}
	\emph{Step 1}. The difference $w=u-v$ satisfies $(\partial_t+\partial_{xxx})w=((w+v)^3 - v^3)_x$, $w(t_1)=0$. For any $x_1<x_2$, take $\phi \in \q C^\infty(\m R)$ increasing such that $\phi(x)= 0$ for $x<x_1$ and $\phi(x)=1$ for $x>x_2$. A formal computation  yields
	\begin{equation}\label{eq:diferenca}
	\frac{1}{2}\int w^2\phi dx = \int_1^t\int  \left(-\frac{3}{2} w_x^2\phi_x - w_xw\phi_{xx} -\epsilon ((w+v)^3-v^3)_x w\phi\right)  dxds
	\end{equation}
	This can be rigorously justified by a regularization process: for any $\delta>0$, take $\psi\in C_c^\infty(\m R)$ such that $\text{supp } \psi\subset [t_1,t_2]$ and $\psi\equiv 1$ over $[t_1+\delta, t_2-\delta]$. Then $\psi w$ solves
\[
	(\partial_t+\partial_{xxx})(\psi w) = \psi' w -\epsilon \psi \left((w+v)^3 - v^3\right)_x\quad \text{ in }\mathcal{D}'(\m R\times\m R).
\]
	
Taking a sequence of mollifiers (in both space and time) $(\rho_\e)_{\e>0}$, one has
\[
	(\partial_t+\partial_{xxx})(\rho_\e \ast (\psi w)) = \rho_\e \ast ( \psi' w ) - \epsilon \rho_\e \ast \left(\psi ((w+v)^3 - v^3)_x\right), \quad t\in [t_1,t_2].
\] 
	Writing $w_\e = \rho_\e * (\psi w)$, one now multiplies the above equation by $w_\e \phi$ and integrates over $[t_1+\delta, t] \times \m R$:
\begin{align}\label{eq:regularizada}
\MoveEqLeft \int w_\e(t)^2 \phi dx - \int w_\e (t_1+\delta)^2 \phi dx = \int_{t_1+\delta}^t\int  \left(-\frac{3}{2} (w_\e)_x^2\phi_x - (\partial_x w_\e) w_\e \partial_{xx} \phi \right) dxds\\
& \qquad +\int_{1+\delta}^t\int  \rho_\e \ast (\psi' w-\epsilon\psi ((w+v)^3-v^3)_x) w_\e\phi dxds
\end{align}

	Using the decay properties of $w$ and $v$ (cf. Lemma \ref{lema:decayE}), one may show that
\begin{gather*}
	w\in L^2(\phi dx), \quad \partial_x w \in L^6(\phi dx),\quad  \partial_x w\in L^2(\phi_x dx), \\
 w \partial_x w \in \q C(\supp \partial_{xx} \phi),\quad \text{uniformly in }t\in [t_1,t_2].
\end{gather*}
	Furthermore, since $w\in \q C([t_1,t_2], L^\infty_{loc}(\m R))$ and $|w(t,x)|\lesssim C\|w\| x^{-1}$ for $x\ge 1$, it is trivial to check, using the dominated convergence theorem, that $w^2\phi \in \q C([t_1,t_2], L^1(\m R))$.
	These bounds are sufficient to show that, when $\e \to 0$ in \eqref{eq:regularizada}, one obtains 
	\begin{align}\label{eq:diferencaquase}
	\MoveEqLeft \frac{1}{2}\int w^2(t)\phi dx - \frac{1}{2}\int w^2(t_1+\delta)\phi dx \\
	= & \int_{t_1+\delta}^t\int  \left(-\frac{3}{2} w_x^2\phi_x - w_xw\phi_{xx} - \epsilon ((w+v)^3-v^3)_x w\phi\right)  dxds.\nonumber
	\end{align}
	Finally, using once again the continuity of $\|w^2(t)\phi\|_{L^1}$, the limit $\delta\to 0$ yields \eqref{eq:diferenca}. 
	
	\emph{Step 2.} Fix $\phi_0 \in \q C^\infty(\m R)$ increasing with  $\phi(x)=0$ for $x<0$ and $\phi(x)=1$ for $x>1$. Define the sequence $\phi_n(x)=\phi_0(1+x/n)$, which satisfies
\[ \phi_n(x)\to 1 \text{ as } n\to \infty,\quad \|(\phi_n)_{xx}\|_1 = \frac{\|(\phi_0)_{xx}\|_1}{n}. \]

	Applying \eqref{eq:diferenca} to $\phi_n$,
	\begin{align*}
	\frac{1}{2}\int w(t)^2\phi_ndx \le \int_1^t \left(\|w_xw\|_{L^\infty} \|(\phi_n)_{xx}\|_1 -\epsilon\int ((w+v)^3-v^3)_x w\phi_ndx\right) ds.
	\end{align*}
	We simplify the last term:
\begin{gather*}
	\int (w^3)_x w \phi_n = \int 3ww_x w^2\phi_n,\quad \int (vw^2)_x w \phi_n = \int (v_xw + 2w_x v)w^2\phi_n \\
	\int (v^2w)_x w\phi = \int 2 v_xv w^2\phi_n + \frac{1}{2}\int v^2 (w^2)_x \phi = \int  v_x v w^2\phi_n - \frac{1}{2}\int v^2 w^2 (\phi_n)_x.
\end{gather*}
	Recall that, from Lemma \ref{lema:decayE}, given $u,v\in E$, $\|uv_x\|_{L^\infty}\lesssim  \|u\|\|v\|/t$. Hence
	\begin{align*}
	\frac{1}{2}\int w(t)^2\phi_ndx  \lesssim \|w\|^2 \|(\phi_n)_{xx}\|_1\log t + \frac{\|u_1\|^2 + \|v\|^2}{t}\int w^2\phi_n dx.
	\end{align*}
	Applying Gronwall's inequality, for some $C>$ and all $t\in [t_1, t_2]$, there hold
\[
	\int w(t)^2\phi_n dx \lesssim  \|w\|^2 \| \partial_{xx} \phi_n \|_1(\log t) \exp ( 2C(\|u_1\|^2 + \|v\|^2)\log t).
\]
	Taking the limit $n\to \infty$ and using Fatou's lemma,
	\begin{align*}
	\int w^2(t) dx &\le \liminf \int w^2(t)\phi_n \\&\lesssim \liminf  \|w\|^2 \|(\phi_n)_{xx}\|_1(\log t) \exp ( 2C(\|u_1\|^2 + \|v\|^2)\log t)= 0.
	\end{align*}
	Thus $w\equiv 0$ and $u \equiv v$.
\end{proof}

Using the forward uniqueness property, we are able to obtain some further continuity information on the solution $u$ we constructed.

\begin{prop}[Continuity properties on $\partial_p \tilde u$] \label{prop:cont_u}
Let $u \in \q E([T_-(u_1),T_+(u_1)])$ be the solution constructed in Proposition \ref{prop:exist_u_E}.

The map $t \mapsto \partial_p \tilde u(t)$ is continuous in weak-$L^2$ on $[T_-(u_1),T_+(u_1)]$, continuous to the right in $L^2$, and it is continuous at $t=1$. 
\end{prop}

\begin{proof}
	\emph{Step 1. Continuity at $t=1$ in $L^2$.} 
	
	Given a sequence $t_k \to 1^+$, since $\tilde{u}\in L^\infty((T,\infty), \dot{H}^1((0,+\infty)))$, one may extract a subsequence so that
\[
	\partial_p\tilde{u}(t_k)\rightharpoonup z, \quad z\in L^2((0,+\infty)).
\]
	The continuity of $t\mapsto \tilde{u}(t)\in L^\infty(\m R)$ implies that $z=\tilde{u}_x(1)$. Since $u\in \q{E}([T,\infty))$, given any compact set $K\subset \m R$, the estimates of Lemma \ref{lema:decayE} imply that $(u(t_k))_{k\in\mathbb{N}}$ is in $W^{1,\infty}(K)$. Hence, by Ascoli-Arzel\'a, up to a subsequence, there exists $v\in \q C(\m R)$ such that
\[
	u(t_k) \to v\quad \text{ in }L^\infty(K),\ K\text{ compact subset of }\m R.
\]
	Moreover, due to the uniform decay of $u(t_k)$, one sees that this convergence is valid over $L^\infty(\m R)$. Since, by Step 3, $\tilde{u}(t_k)\to \tilde{u}(1)$ in $L^\infty(\m R)$, one must have $v=u(1)$. Together with the decay estimate \eqref{eq:est1lem1}, the Dominated Convergence theorem implies that
\[
	\tilde{u}(t_k)\to \tilde{u}(1) \quad \text{in }L^6(\m R).
\]
	Hence
	\begin{align*}
	e^{-it_kp^3}\widehat{\mathcal{I}u}(t_k)&=i\left(\partial_p\tilde{u}(t_k) - \frac{3t_k}{4\pi}e^{it_kp^3}\widehat{u^3}(t_k)\right) \rightharpoonup e^{-ip^3}\widehat{\mathcal{I}u}(1)\quad \text{ in }L^2((0,+\infty)). 
	\end{align*}
	On the other hand, using \eqref{eq:H1aprox_tmaior} from Lemma \ref{lem:H1aprox},
	\begin{align*}
	\|(\mathcal{I}u)(t_k)\|_{L^2((0,+\infty))}&\lesssim \lim_n \left(\int_0^\infty |\widehat{\mathcal{I}u}(t_k)|^2\chi_n^{-1}dp\right)^{1/2}  \\ & \lesssim \lim_n \left(\int_0^\infty |\widehat{\mathcal{I}u}(1)|^2\chi_n^{-1}dp + o_n(1)\delta^3 \right) t_k^{1/6} \\&= \|(\mathcal{I}u)(1)\|_{L^2((0,+\infty))}t_k^{1/6}.
	\end{align*}
	Therefore $e^{-it_kp^3}\widehat{\mathcal{I}u}(t_k)\to e^{-ip^3}\widehat{\mathcal{I}u}(1)$ in $L^2((0,+\infty))$ and
	\begin{equation}
	\partial_p\tilde{u}(t_k) = -ie^{-it_kp^3}\widehat{\mathcal{I}u}(t_k) + \frac{3t_k}{4\pi}e^{it_kp^3}\widehat{u^3}(t_k) \to \partial_p\tilde{u}(1) \quad \text{in } L^2((0,+\infty)).
	\end{equation}
The continuity to the left of $t=1$ follows from the same arguments and \eqref{eq:H1aprox_tmenor}.
	
\bigskip

\emph{Step 2. Continuity to the right in $L^2$.}
	
Now we observe that $\tilde{u}$ is continuous to the right with values in $\dot{H}^1((0,+\infty))$. 
	
Given $t_0\in [T_-(u_1),T_+(u_1))$, Proposition \ref{prop:exist_u_E} show that one may build $v$ satisfying $\tilde{v}\in \q E([t_0-\e,t_0+\e])$, for some $\e >0$, with $v$ a solution to \eqref{mkdv} in $\q D'((t_0-\e, t_0+\e) \times \m R)$ and $v(t_0)=u(t_0)$.
	
Step 1 shows that we can furthermore assume continuity at $t_0$: 
\[ \partial_p \tilde{v}(t) \to \partial_p \tilde{v}(t_0) \quad \text{in  }L^2((0,+\infty)) \quad \text{as} \quad t\to t_0. \]
By forward uniqueness, $v\equiv u$ on $[t_0,t_0+\e]$, which means that $\partial_p \tilde{u}$ is continuous to the right at $t_0$.
	
\bigskip
	
	\emph{Step 3. Continuity in weak-$L^2$.
	} 

Let $(t_n)_{n\in \mathbb{N}}$ be a sequence of times in $[T_-(u_1),T_+(u_1)]$ such that $t_n \to t_*$. We already saw that $\tilde u(t_n) \to \tilde u(t_*)$ in $L^\infty$. Since $\partial_{p}\tilde{u}(t_n)$ in bounded in $L^2((0,+\infty))$, any subsequence admits a sub-subsequence converging in weakly in $L^2$, to a limit which can only be $u(t_*)$. This proves that the full sequence converges: $\partial_p \tilde u(t_n) \weak \partial_p u(t_*)$ weakly in $L^2$.
\end{proof}

\begin{nb}
If backward uniqueness holds, then the same proof shows full continuity: $\partial_p \tilde u \in \q C([T_-(u_1),T_+(u_1)],L^2)$.
\end{nb}

\begin{prop}[Forward uniqueness implies a backward blow-up alternative]\label{prop:blowupaltern}
	Let $u\in \q E((T_-, T_+))$ be a maximal solution of (mKdV). If $T_->0$, then
\[
	\limsup_{t\to T_-} \|u(t)\|_{\q E(t)}=\infty.
\]
\end{prop}

\begin{proof}
	By contradiction, suppose that
\[
	\forall t \in (T_-,1), \quad \|u(t)\|_{\q E(t)}<M.
\]
	It follows from Step 3, proof of Proposition \ref{prop:exist}, that $\tilde{u}$ is Lipschitz in time with values in $L^\infty(\m R)$. Moreover, since $T_->0$, $\tilde{u}$ is bounded in $\dot{H}^1((0,+\infty))$. This means that $u$ may be extended up to $t=T_-$: $u\in \q E([T_-, 1])$.

Now consider the solution $v\in \q E([T_- - \e, T_- + \e)$ of (mKdV) with $v(T_-)=u(T_-)$. By forward uniqueness, $u\equiv v$ for $t>T_-$. This means that $u$ is not maximal, a contradiction.
\end{proof}

\section{Well-posedness on $L^2(\m R)\cap\q E$} \label{sec:6}

In this section, we prove that Proposition \ref{prop:L2}, that is, once we restrict the critical space to $L^2$-integrable functions, the local well-posedness theory works in either direction of time. We split it into two statements: one for uniqueness and another for persistence of $L^2$ integrability.

\begin{prop}[Backward uniqueness]\label{prop:backuniqL2}
	If $u, v \in \q E([t_1,t_2])$ are two solutions of \eqref{mkdv} with $u, v \in \q C([t_1,t_2], L^2(\m R))$ and $u (t_2)= v(t_2)$, then $u_1\equiv u_2$.
\end{prop}
\begin{proof}
	Set $w=u - v$. Then $(\partial_t+\partial_{xxx})w=- \epsilon (u^3 - v^3)_x$. The right-hand side of the equation is in $L^\infty((t_1,t_2),L^2(\m R))$:
\[
	\|(u^3)_x(t)\|_{L^2}\lesssim \|u_x u(t)\|_{L^\infty}\|u(t)\|_{L^2} \lesssim \|u\|_{\q E(t_1,t_2)}\|u\|_{L^\infty((t_1,t_2),L^2)}.
\]
	Applying $w$ to the equation, we see that
	
	\begin{align*}
	\left|\frac{1}{2}\frac{d}{dt}\|w(t)\|_2^2\right|& \lesssim (\|u_x u(t)\|_{L^\infty} + \| v_x v (t)\|_{L^\infty}) \|w(t)\|_{L^2}^2\\&
	\lesssim (\|u\|_{\q E(t_1,t_2)} + \|u\|_{\q E(t_1,t_2)})\|w(t)\|_{L^2}^2.
	\end{align*}
	
	By Gronwall's lemma, we obtain $w\equiv 0$.
\end{proof}

We now prove existence of solutions in   $L^2(\m R)\cap\q E$, which can be translated into a persistence result:

\begin{prop}[Persistence of $L^2$ integrability]\label{prop:persistL2}
	Given $u_1\in \q E(1)\cap L^2(\m R)$,  consider the corresponding solution $u\in \q E(I)$ of \eqref{mkdv} given by Proposition \ref{prop:exist}. Then $u\in \q C(I, L^2(\m R))$.
\end{prop}
\begin{proof}

	Consider the approximate solutions $u_n$ of \eqref{pimkdv}. Since $u_n\chi_{n}^{-1}\in \q C(I, L^{\infty}(\m R))$, we have $\tilde{u}_n\chi_{n}^{-1/2} \in \q C(I,L^2(\m R))$.
	It then follows by direct integration that
	\begin{align*}
	\left|\frac{1}{2}\frac{d}{dt}\int |\tilde{u}_n(t,p)|^2\chi_{n}^{-1}dp\right| &= \left|\int \widehat{(u_n^3)_x}(t,p)\hat{u}_n(t,p)dp\right|=  \left|\frac{1}{2}\int (u_n^2)_xu_n^2dx\right|=0
	\end{align*}
	and  so $\|\tilde{u}\chi_n^{-1/2}(t)\|_{L^2(\m R)}^2$ is conserved. 
	In the limit $n\to \infty$, we obtain $u\in L^\infty(I, L^2(\m R))$. We then infer $L^2$-conservation by direct integration of the equation for $u$:
	\[	\|u(t)\|_{L^2}=\|u_1\|_{L^2}. \]
	Together with the weak $L^2$-continuity, we conclude that $u\in \q C(I, L^2(\m R))$.
\end{proof}

\section{Construction of blow-up solutions} \label{sec:7}

\begin{proof}[Proof of Proposition \ref{prop:blowupgivenprofile}]
	
\textit{Step 1. Construction of the approximating sequence.} The scaling invariance of the equation and the criticality of the space $\q E$ imply that Proposition \ref{prop:exist} holds for any initial time $t_0>0$, with $\delta$ independent on $t_0$.

Given $t_n\to 0$, define
\[
g_n(p)=g_0(t_n^{1/3}p).
\]
Since
\[
\|g_n\|_{\q E(t_n)}=\|g_0\|_{\q E(1)}<\delta,
\]
one may build $u_n\in \q E([t_n,+\infty))$ solution of (mKdV) with initial condition $\tilde{u}_n(t_0)=g_n$ and
\[ \|u_n\|_{\q E([t_n, +\infty))}\le C\delta. \]

\textit{Step 2. Convergence.} Proceeding as in Steps 2-5 of the proof of Proposition \ref{prop:exist}, we obtain a solution $u\in \q E((0,+\infty))$ of (mKdV) which is the pointwise limit of $(u_n)_{n\in \m N}$ both in the physical and frequency spaces.

Given $t_0>0$, one may construct a solution $v\in \q E([T_-(v(t_0)),+\infty))$ of \eqref{mkdv} with $v(t_0)=u(t_0)$. By uniqueness, $u\equiv v$ for all $t\ge t_0$. Since $t_0$ is arbitrary, this implies $u\in \q E((0,+\infty))$.

\textit{Step 3. Behavior at $t=0$.} Define
\[
g_n(t,p)=\left\{\begin{array}{lc}
\tilde{u}_n(t,t^{-1/3}p)& 1\ge t>t_n\\
g_0(p)&t_n\ge t\ge 0
\end{array}\right..
\]
Hence, for $t>t_n$,
\[
\partial_tg_n(t,p) = (\partial_t\tilde{u}_n)(t,t^{-1/3}p) -\frac{p}{3t^{4/3}}(\partial_p\tilde{u}_n)(t,t^{-1/3}p)=-\frac{3p}{t}\widehat{\mathcal{I}u_n}(t,t^{-1/3}p).
\]
Remark \ref{rmk:boundI} implies
\[
\left\|\frac{1}{p}\partial_tg_n(t)\right\|_{L^2((0,+\infty))}\lesssim \frac{1}{t}\|\widehat{\mathcal{I}u_n}(t,t^{-1/3}p)\|_{L^2((0,+\infty))}\lesssim t^{-5/6}\|\widehat{\mathcal{I}u_n}(t)\|_{L^2((0,+\infty))} \lesssim \delta t^{-2/3}.
\]
As a consequence, the sequence $(g_n)_{n\in\m N}$ is uniformly continuous in time: for any $0\le t,s\le 1$,
\[
\left\| \frac{1}{\jap{p}}(g_n(t,p))-g_n(s,p))\right\|_{L^2((0,+\infty))}\le \int_s^t \left\|\frac{1}{p}\partial_t g_n(s') \right\|_{L^2((0,+\infty))} ds' \lesssim \delta|t^{1/3}-s^{1/3}|.
\]
Since
\[
\left\| \frac{1}{\jap{p}}g_n(t,p)\right\|_{L^2(\m R)}\lesssim \|g_n(t)\|_{L^\infty(\m R)}\lesssim \delta,
\]
by the Ascoli-Àrzela theorem, there exists $g \in \q C([0,1], L^2(\jap{p}^{-1}dp))$ such that
\[
g_n \to g  \mbox{ in }\q C([0,1], L^2(\jap{p}^{-1}dp)),\quad g(0,p)=g_0(p)
\]
and
\[
\left\| \frac{1}{\jap{p}}(g (t,p))-g_0(p))\right\|_{L^2(\m R)} \lesssim \delta t^{1/3}, t>0.
\]
On the other hand, for each fixed $t>0$,
\[
g_n(t,p)=\tilde{u}_n(t,t^{-1/3}p) \to \tilde{u}(t,t^{-1/3}p) \quad \text{in } L^\infty(\m R) \text{ as } n\to \infty.
\]
Thus $\tilde{u}(t,t^{-1/3}p)=g(t,p)$ and \eqref{eq:taxablowup} follows.
\end{proof}

\section{Asymptotic behaviour as $t\to+\infty$} \label{sec:8}

Given $\nu\in (0,1/2)$, consider the norm
\[
\|u\|_{\mathcal{Y}_t^\nu}= t^{\frac{\nu}{3}-\frac{1}{6}}\|\partial_p\tilde{u}\|_{L^2} + \sup_{p\in\m R}\left\{ |p|^{-\nu}\jap{p^3t}^{\frac{\nu}{3}-\frac{1}{6}}|\tilde{u}(p)| \right\}
\]
and define
\[
\mathcal{Y}_t^\nu=\left\{ u\in \mathcal{S}'(\m R;\m R)\ : \tilde{u}(0)=0, \ \|u\|_{\mathcal{Y}_t^\nu}<\infty \right\}.
\]

For the sake of completeness, we recall the following technical lemma.
\begin{lem}[{\cite[Lemma 2.2]{HN01}}]\label{lem:media0}
	Given $\nu\in (9/20, 1/2)$, let $u,v, w\in \q E([1,T])$ be such that $\tilde{w}(t,0)=0$.
	Then there exists a universal constant $C>0$ such that
	\begin{equation}\label{eq:trilinear}
	\forall t \in [1,T], \quad \|u(t)v(t)w(t)\|_{L^2}\le C t^{-\frac{5}{6}-\frac{\nu}{3}}\|u(t) \|_{\q E(t)} \|v(t)\|_{\q E(t)} \|w(t)\|_{\mathcal{Y}_t^\nu}.
	\end{equation}
	Furthermore,
	\begin{equation}\label{eq:asympw}
	\left|w(t,x)-\frac{1}{t^{1/3}}\partere \Ai\left(\frac{x}{t^{1/3}}\right)\tilde{w}(t,y)\right|\lesssim t^{-1/3-\nu/3}\jap{|x|/t^{1/3}}^{-1/4}\|w(t)\|_{\mathcal{Y}_t^\nu},
	\end{equation}
	where
	\begin{equation}\label{eq:p0}
	y=\left\{\begin{array}{ll}
\sqrt{-x/3t},& x<0\\
0,& x>0
\end{array}\right..
	\end{equation}

\end{lem}

\begin{prop}[Asymptotics on the Fourier space]\label{prop:asympfourier}
	Given $u\in \q E([1,\infty))$ solution of $(\partial_t+\partial_{xxx})u=(u^3)_x$ with $\|u\|_{\q E([0,+\infty))}\le \delta$, let $S\in \q E([1,\infty))$ be a self-similar solution with
\[
	\widehat{S(1)}(0^+)=\widehat{u(1)}(0^+).
\]
	Then, for any $\nu\in (9/20,1/2)$,
	\begin{equation}\label{eq:restoemYt}
	\|u(t)-S(t)\|_{\mathcal{Y}_t^\nu}<30\delta,\quad t\ge 1.
	\end{equation}
	On the other hand, there exists $U_\infty\in C_b(\m R\setminus\{0\})$ such that
	\begin{equation}\label{eq:modificadofourier}
	\left| \tilde{u}(t,p) - U_\infty(p)\exp\left(-\frac{i\epsilon}{4\pi}|U_\infty(p)|^2\log t\right) \right|\lesssim \frac{\delta}{\jap{p^3t}^{\frac{1}{12}}}
	\end{equation}
	Finally, one has
\begin{equation}\label{eq:relacaoSU}
	\lim_{p\to 0} |U_\infty(p)| = \lim_{tp^3\to \infty} |\mathcal{S}(t,p)|. 
\end{equation}
	
\end{prop}

\begin{nb}
	As a direct consequence, we see that, if $S$ and $S'$ are two self-similar solutions with $\|S\|_{\q E(1)}, \|S'\|_{\q E(1)}\le \delta$ and 
\[ \widehat{S(1)}(0)=\widehat{S'(1)}(0), \]
	then $S=S'$.
\end{nb}

\begin{proof}
	\emph{Proof of }\eqref{eq:restoemYt}. 
	Set $w(t)=u(t)-S(t)$. Observe that
\[
	\|w(1)\|_{\mathcal{Y}_t^\nu}\le 2\delta.
\]
	Suppose that there exists $T_1>1$ such that
\[
	\|w(t)\|_{\mathcal{Y}_t^\nu}<30\delta,\quad 1\le t\le T_1,\text{ and }\|w(T_1)\|_{\mathcal{Y}_{T_1}^\nu}=30\delta.
\]
	Due to the self-similar structure of $S$, $\mathcal{I}S=0$ and so
\[
	\partial_p\widetilde{S}=3te^{itp^3}\widehat{S^3}.
\]
	Since $(\partial_t+\partial_{xxx})w=-\epsilon(w^3-3uw^2+3u^2w)_x$,
	\begin{align*}
	\|\tilde{w}_p(t)\|_{L^2((0,+\infty))}&\le \|\widehat{\mathcal{I}w}(t)\|_{L^2((0,+\infty))} + 3t\|(w^3-3uw^2+3u^2w)(t)\|_{L^2} \\&\le \|\widehat{\mathcal{I}u}(t)\|_{L^2((0,+\infty))} + 30\delta^3 t^{\frac{1}{6}-\frac{\nu}{3}} \le 2\delta t^{\frac{1}{6}-\frac{\nu}{3}}.
	\end{align*}
	Moreover, since $\tilde{w}(0)=0$,
	\begin{align*}
	|\tilde{w}(t,p)|\le \sqrt{p}\|\tilde{w}_p(t)\|_{L^2} \le 2\delta |p|^{\nu}\jap{p^3t}^{\frac{1}{6}-\frac{\nu}{3}}
	\end{align*}
	Thus 
\[
	30\delta=\|w(T_1)\|_{\mathcal{Y}_{T_1}^\nu}\le 4\delta,
\]
	which is a contradiction. Hence 
\[
	\|w(t)\|_{\mathcal{Y}_t^\nu}<30\delta,\quad t\ge 1
\]
	and \eqref{eq:restoemYt} follows. 
	
	\emph{Proof of }\eqref{eq:modificadofourier}. From \eqref{eq:decayuniformt}, we have 
\[
	|(\tilde{u}E_u)(t,p)-(\tilde{u}E_u)(\tau,p)|\lesssim \frac{\delta}{\jap{p^3\tau}^{\frac{1}{12}}},\quad t\ge \tau
\]
	where we recall that  (see )
\[
	E_u(t,p)=\exp\left(-i\epsilon\int_1^t\frac{p^3}{4\pi \jap{p^3s}}|\tilde{u}(s,p)|^2ds\right).
\]
	Thus there exists $U\in L^\infty(\m R)$, $U\in \q C(\m R\setminus\{0\})$, such that
	\begin{equation}\label{eq:perfilfinal}
	|(\tilde{u}E_u)(t,p)- U(p)|\lesssim \frac{\delta}{\jap{p^3t}^{\frac{1}{12}}}.
	\end{equation}
	Writing
\[
	\psi(t,p)=\int_1^t\frac{p^3}{4\pi \jap{p^3s}}(|\tilde{u}(s,p)|^2 - |\tilde{u}(t,p)|^2)ds,
\]
	we have
	\begin{align*}
	\psi(t)-\psi(\tau)&=\int_\tau^t \frac{p^3}{4\pi \jap{p^3s}}(|\tilde{u}(s,p)|^2 - |\tilde{u}(t,p)|^2)ds \\
	& \qquad + (|\tilde{u}(\tau,p)|^2 - |\tilde{u}(t,p)|^2)\int_\tau^t \frac{p^3}{4\pi \jap{p^3s}}ds 
	\end{align*}
	which implies that
\[
	|\psi(t,p)-\psi(\tau,p)|\lesssim \frac{\delta}{\jap{p^3\tau}^{\frac{1}{12}}}
\]
	Therefore there exists $\Psi\in L^\infty(\m R)$, $\Psi\in \q C(\m R\setminus\{0\})$ such that
	\begin{equation}\label{eq:fasefinal}
	|\psi(t,p)-\Psi(p)|\lesssim \frac{\delta}{\jap{p^3t}^{\frac{1}{12}}}
	\end{equation}
	We decompose
	\begin{align*}
	\int_{1}^t\frac{p^3ds}{\jap{p^3s}} &= \int_{p^3}^{p^3t}\frac{ds}{\jap{s}} = \int_{p^3}^{p^3t}\frac{ds}{s} +\int_{p^3}^{p^3t}\left(\frac{1}{\jap{s}}-\frac{1}{s}\right)ds \\&= \log t + \int_{p^3}^\infty \frac{s-\jap{s}}{s\jap{s}} -  \int_{tp^3}^\infty \frac{s-\jap{s}}{s\jap{s}} =:\log t + \mathcal{R}(p^3) - \mathcal{R}(tp^3)
	\end{align*}
	Since
\[
	|\jap{s}-s|=s\left(\left(1+\frac{1}{s^2}\right)^{1/2}-1\right)\lesssim s^{-1},
\]
	we have
\[
	|\mathcal{R}(y)|\lesssim \frac{1}{y^2}.
\]
	Thus, 
\[
	\left|\int_1^t \frac{p^3 |\tilde{u}(s,p)|^2ds}{4\pi(p^3s)^{1-\gamma/2} \jap{p^3s}^{\gamma/2}} - \Psi(p)-\frac{1}{4\pi}|U(p)|^2\mathcal{R}(p^3) - \frac{1}{4\pi}|U(p)|^2\log t\right| \lesssim \frac{\delta}{(p^3t)^{\frac{1}{12}}}.
\]
	Defining
\[
	U_\infty(p)=U(p)\exp\left(-i\epsilon\Psi(p) - \frac{i\epsilon}{4\pi}|U(p)|^2\mathcal{R}(p^3)\right),
\]
	\eqref{eq:perfilfinal} and \eqref{eq:fasefinal} imply
	\begin{align*}
	\left|\tilde{u}(t,p) - U_\infty(p)\exp\left(- \frac{i\epsilon}{4\pi}|U_\infty(p)|^2\log t\right) \right|\lesssim \frac{\delta}{\jap{p^3t}^{\frac{1}{12}}}.
	\end{align*}
	
	\emph{Proof of }\eqref{eq:relacaoSU}: Using \eqref{eq:restoemYt} and \eqref{eq:modificadofourier},
\[
	\left| |S(t,p)| - |U_\infty(p)| \right|\lesssim |p|^{\nu}\jap{tp^3}^{\frac{1}{6}-\frac{\nu}{3}} + \frac{1}{\jap{p^3t}^{\frac{1}{12}}}
\]
	Take, at the same time, $tp^3\to \infty$ and $p\to 0$ so that the right-side goes to zero. Then
\[
	\lim_{p\to 0} |U_\infty(p)| = \lim_{tp^3\to \infty} |S(t,p)|.
\]
\end{proof}

\begin{nb}
	If $u$ is an $L^2$ solution of (mKdV), then \eqref{eq:modificadofourier} implies that $U_\infty\in L^p(\m R)$, for $p$ large. On the other hand, if one takes $u$ as a self-similar solution, then $|U_\infty|$ is a constant.
\end{nb}

\begin{cor}[Asymptotics in the physical space]
	Given $u\in \q E([1,\infty))$ solution of $(\partial_t + \partial_{xxx})u=(u^3)_x$ with $\|u\|_{\q E([1,+\infty)}\le \delta$, let $S\in \q E([1,\infty))$ be the self-similar solution with
\[
	\widehat{S(1)}(0^+)=\widehat{u(1)}(0^+).
\]
	Then, for any $\nu\in(9/20,1/2)$,
	\begin{equation}\label{eq:asympu}
	\left\|u(t)-S(t)\right\|_{L^\infty}\lesssim  \frac{\delta}{t^{1/3+\nu/3}}.
	\end{equation}
	On the other hand,
	\begin{equation}\label{eq:asympumodified}
	\left|u(t,x)-\frac{1}{t^{1/3}}\partere\Ai\left(\frac{x}{t^{1/3}}\right)U_\infty\left(y\right)\exp\left(-\frac{i\epsilon}{6}|U_\infty(y)|^2\log t\right)\right|\le  \frac{\delta}{t^{1/3}\jap{x/t^{1/3}}^{3/10}},
	\end{equation}
	where $y$ is defined by \eqref{eq:p0}.
\end{cor}

\begin{proof}
	Define $w(t):=u(t,x)-S(t,x)$. Then, by Proposition \ref{prop:asympfourier},
\[
	\|w(t)\|_{\mathcal{Y}_t^\nu}<30\delta.
\]
	The definition of the norm of $\mathcal{Y}_t^\nu$ and the decay estimate for the Airy-Fock function imply
\[
	\left|\frac{1}{t^{1/3}}\partere \Ai\left(\frac{x}{t^{1/3}}\right)\tilde{w}\left(t,\sqrt{\frac{|x|}{3t}}\right)\right|\lesssim \delta t^{-1/3}\jap{x/t^{1/3}}^{-1/4}\left|x/t\right|^{\nu/2}\jap{x^{3/2}/t^{1/2}}^{\frac{1}{6}-\frac{\nu}{3}}\lesssim \frac{\delta}{t^{1/3+\nu/3}}.
\]
	Together with \eqref{eq:asympw}, we obtain \eqref{eq:asympu}. Finally, \eqref{eq:asympumodified} follows from \eqref{eq:est6lem1} and \eqref{eq:modificadofourier}.
\end{proof}

\section{Well-posedness for perturbations of self-similar solutions} \label{sec:9}

We remark that the results of section \ref{sec:6} exclude nontrivial self-similar solutions. On the other hand, the lack of backwards uniqueness in $\q E$ is especially problematic for the study of the dynamics around self-similar solutions at time $t=0$. To overcome this, we consider the space $\q R_\alpha$ defined in \eqref{def:R_a} in Section 2.

First let us observe that $\q R_\alpha$ is not empty. If one considers $\alpha=1$ and $a_k=1$, $k\ge 1$, then any $L^2$-function $f$ with $\mbox{supp} \hat{f} \subset B_{R}(0)$, $R<1$, is in $\q R_\alpha$:
	\[
	\|\partial_x^k f\|_{L^2}^2 = \|p^k \hat{f}\|_{L^2}^2\lesssim  R^k \| \hat{f}\|_{L^2}^2 \le \| \hat f \|_{L^2}^2.
	\]
	This means that low-frequency perturbations of self-similar solutions are acceptable.

\begin{prop}[Backward uniqueness]
	If $u_1,u_2\in \q{E}(t_1,t_2)$ are two solutions of \eqref{mkdv} with $u_1-u_2\in L^\infty((t_1,t_2),\q R_\alpha)$ and $u_1(t_2)=u_2(t_2)$, then $u_1\equiv u_2$.
\end{prop}
\begin{proof}
	Having the \textit{a priori} knowledge that $u_1-u_2\in L^\infty((t_1,t_2),H^1(\m R))$ allows us to proceed as in Proposition \ref{prop:uniq}.
\end{proof}

Now fix a self-similar profile $\mathcal{S}\in \q E(1)$ and $\tilde{S}(t,p)=\mathcal{S}(t^{1/3}p)$. Define
\begin{equation}
\q E_S(t) = \left\{ u\in \q E(t): \hat{u}(t)-\tilde{S}(t)\in \q R_\alpha  \right\}
\end{equation}
and, for any time interval $I\subset(0,\infty)$,
\begin{equation}
\q E_S(I)=\left\{ u\in \q E(I): \hat{u}-\tilde{S}\in \q C(I,\q R_\alpha) \right\}
\end{equation}
endowed with the norm
\[
\|u\|_{\q E_S(I)}=\sup_{t \in I} \left(\|u(t)\|_{\q E(t)} + \|\hat{u}(t)-\tilde{S}(t)\|_{\q R_\alpha}\right),
\]
The remainder of this section concerns the existence of solution on $\q E_S$, which follows from a persistence result analogous to Proposition \ref{prop:persistL2}.

\begin{lem}\label{lem:uniqS}
	Let $v$ be a solution of \eqref{mkdv} given by Proposition \ref{prop:exist} with initial condition $\tilde{v}(1)=\mathcal{S}$. Then $v\equiv S$.
\end{lem}
\begin{proof}
	By forward uniqueness, $v\equiv S$ for $t>1$. Since $S$ is a self-similar solution, it satisfies
	\[
	\|\widehat{\mathcal{I}v}(t)\|_{L^2((0,+\infty))}=	\|\widehat{\mathcal{I}S}(t)\|_{L^2((0,+\infty))}=0, \ t\ge 1.
	\]
	The inequality \eqref{eq:estIparatras} then implies that $\widehat{\mathcal{I}v}(t,p)\equiv 0$, $t<1, p\neq0$. Setting
	\[
	\mathcal{V}(t,p)=\tilde{v}(t,t^{-1/3}p)
	\]
	a simple computation yields
	\[
	\partial_t\mathcal{V}(t,p)=-\frac{3p}{t}\widehat{\mathcal{I}v}(t, t^{-1/3}p)\equiv 0.
	\]
	Hence $\mathcal{V}(t,p)=\mathcal{V}(1,p)=\mathcal{S}(p)$ and so $v\equiv S$.
\end{proof}

For any given $n\in\m N$, we define $S_n$ as the solution of \eqref{pimkdv} with $\tilde{u}_1=\mathcal{S}$. For large $n$, Lemma \ref{lem:uniformtime} implies that $S_n$ is defined on $[T_-,T_+]$ and Lemma \ref{lem:uniqS} ensures that the limit of $S_n$ is the self-similar solution $S$.

\begin{prop}[Persistence of $\q E_S$]
	Given $u_1\in \q E_S(1)$, the corresponding solution $u\in \q E(I)$ of \eqref{mkdv} given by Proposition \ref{prop:exist} satisfies $u\in \q E_S(I)$.
\end{prop}

\begin{proof}
	It suffices to consider $I$ bounded. We take the solutions $u_n$ of the approximate equation \eqref{pimkdv}. Define $w_n=u_n-S_n$. It follows directly that $\partial_x^kw_n\in  \q C(I,\q L^2)$, for any $k$.  Then, integrating the equation for $w$,
	\begin{align*}
	\left|\frac{1}{2}\frac{d}{dt}\|\partial_x^kw_n\|_{L^2}^2 \right|&\lesssim \left|\int \partial_x^k\left(\Pi_n((u_n^3-S_n^3)_x)\right)\partial_x^k w_n dx \right|\\&\lesssim (\|(u_n)^2\|_{W^{1,\infty}} + \|(S_n)^2\|_{W^{1,\infty}} )\|w_n\|_{H^1}\|\Pi_n\partial_x^{2k} w_n\|_{L^2} \\&\lesssim (\|u\|_{\q E(I)}^2 + \|S\|_{\q E(I)}^2)\|w_n\|_{\q R_\alpha}\|\partial_x^{2k} w_n\|_{L^2}.
	\end{align*}
	Hence
	\begin{align*}
	\MoveEqLeft a_k \|\partial_x^k w_n(t)\|_{L^2}^2 \le a_k\|\partial_x^k w_n(1)\|_{L^2}^2 + C\int_1^t(\|u\|_{\q E(I)}^2 + \|S\|_{\q E(I)}^2)\|w_n(s)\|_{\q R_\alpha}a_k\|\partial_x^{2k} w_n(s)\|_{L^2}ds \\ 
	&\le a_k\|\partial_x^k w_n(1)\|_{L^2}^2 + C\alpha\int_1^t(\|u\|_{\q E(I)}^2 + \|S\|_{\q E(I)}^2)\|w_n(s)\|_{\q R_\alpha}a_{2k}\|\partial_x^{2k} w_n(s)\|_{L^2}ds.
	\end{align*}
	Taking the supremum in $k$ and applying Gronwall's lemma,
	\[
	\|w_n(t)\|_{\q R_\alpha}^2\le  \|w_n(1)\|_{\q R_\alpha}^2\exp\left(C\alpha(\|u\|_{\q E(I)}^2 + \|S\|_{\q E(I)}^2)|t-1|\right)
	\]
	uniformly on $n$. Taking the limit $n\to \infty$, $u-S\in L^\infty(I, \q R_\alpha)$. The uniqueness in $\q E_S$ reduces the continuity at any time to the continuity at $t=1$, which follows from the above inequality. Indeed,
	\begin{align*}
	\lim_{t\to 1}\|w(t)\|_{\q R_\alpha}^2 &\le \lim_{t\to 1}\|w_n(t)\|_{\q R_\alpha}^2 \le \lim_{t\to 1}\|w_n(1)\|_{\q R_\alpha}^2\exp\left(C\alpha(\|u\|_{\q E(I)}^2 + \|S\|_{\q E(I)}^2)|t-1|\right) \\
&= \|w_n(1)\|_{\q R_\alpha}^2,
	\end{align*}
	and, taking the limit on the right-hand side,
	\[
	\lim_{t\to 1}\|w(t)\|_{\q R_\alpha}^2\le \|w(1)\|_{\q R_\alpha}^2.
	\]
	Since $w(t)\rightharpoonup w(1)$ in $\q R_\alpha$, we obtain strong convergence in $\q R_\alpha$.
\end{proof}

We are now in a position to prove Proposition \ref{prop:stab} stated in Section 2.

\begin{proof}[Proof of Proposition \ref{prop:stab}]
	Suppose that $u=S+w$, with $S$ self-similar and $w\in \q R_\alpha$, is defined on an interval $I\subset (0,1]$. Then, setting $z:=u^2+uS + S^2$, $(\partial_t + \partial_{xxx})w= -\epsilon (zw)_x$ and, by Sobolev embedding,
	\[
	\|z\|_{L^\infty}\lesssim \|w\|_{L^\infty}^2 + \|S\|_{L^\infty}^2 \lesssim \|w\|_{\q R_\alpha}^2 + t^{-2/3}\|S\|_{\q E(1)}^2.
	\]
	
	The following formal computations can be made rigorous by approximating with solutions of \eqref{pimkdv}. Integrating directly the equation for $w$,
	
	\begin{align*}
	\left|\frac{d}{dt}\|\partial_x^k w\|_{L^2}^2\right|&\lesssim \left|\int \partial_x^{k+1}(zw)\partial_x^kwdx\right|\lesssim \|z\|_{L^\infty}\|w\|_{L^2}\|\partial_x^{2k+1} w\|_{L^2}\\
	&\lesssim (\|w\|_{\q R_\alpha(I)}^2 + t^{-2/3}\|S\|_{\q E(1)}^2)\|w\|_{L^2}\|\partial_x^{2k+1} w\|_{L^2}
	\end{align*}
	and so, for any $t\in I$,
	\begin{align*}
	a_k\|\partial_x^k w(t)\|_{L^2}^2&\lesssim \|w(1)\|_{\q R_\alpha}^2 + \frac{a_k}{a_{2k+1}}(\|w\|_{\q R_\alpha(I)}^2 + \|S\|_{\q E(1)}^2)\|w\|_{\q R_\alpha(I)}^2\\&\lesssim \|w(1)\|_{\q R_\alpha}^2 + \alpha(\|w\|_{\q R_\alpha(I)}^2 + \|S\|_{\q E(1)}^2)\|w\|_{\q R_\alpha(I)}^2.
	\end{align*}
	These estimates imply
	\begin{align*}
	\|w\|_{\q R_\alpha(I)}^2\lesssim \|w(1)\|_{\q R_\alpha}^2 + \alpha(\|w\|_{\q R_\alpha(I)}^2 + \|S\|_{\q E(1)}^2)\|w\|_{\q R_\alpha(I)}^2.
	\end{align*}
	Therefore, if $\|w(1)\|_{\q R_\alpha}^2, \alpha\|S\|_{\q E(1)}^2< \delta$ sufficiently small, 
	\[
	\| w \|_{\q R_\alpha(I)}^2\lesssim \delta.
	\]
	On the other hand, since $\mathcal{I}w=\mathcal{I}u$ and $(\partial_t + \partial_{xxx})\mathcal{I}u = -3 \epsilon u^2(\mathcal{I}u)_x$,
	\[
	\left|\frac{1}{2}\frac{d}{dt}\|\mathcal{I}w\|_{L^2}^2\right|\lesssim \left| \int u^2(\mathcal{I}w)_x\mathcal{I}w \right| \lesssim \|uu_x\|_{L^\infty}\|\mathcal{I}w\|_{L^2}^2 \lesssim (\|w\|_{\q R_\alpha(I)}^2 + \frac{1}{t}\|S\|_{\q E(I)}^2)\|\mathcal{I}w\|_{L^2}^2.
	\]
	The integration of this inequality implies that $ \|\mathcal{I}u\|_{L^2}^2$ cannot blow-up in $I$. Finally, since 
	\[ \partial_p \tilde{w} = \widehat{\mathcal{I}w} + 3t \widehat{zw}, \]
	we estimate
	\begin{align*}
	\|\tilde{w}\|_{L^\infty}&\lesssim \|\tilde{w}\|_{L^2}^{1/2}\|\partial_{p}\tilde{w}\|_{L^2}^{1/2} \lesssim \|w\|_{\q R_\alpha(I)}^{1/2}\left(\|\mathcal{I}w\|_{L^2} + t\|z\|_\infty\|w\|_{L^2}\right)^{1/2}\\&\lesssim \|w\|_{\q R_\alpha(I)}^{1/2}\left(\|\mathcal{I}w\|_{L^2} + t(\|w\|_{\q R_\alpha}^2 + t^{-2/3}\|S\|_{\q E(1)}^2)\|w\|_{L^2}\right)^{1/2}.
	\end{align*}
	By Proposition \ref{prop:blowupaltern}, $u$ cannot blow-up at any $t>0$ and the result follows.
\end{proof}

\begin{nb}
	As shown in Lemma \ref{lema:decayE}, the space $\q E$ controls $\|w^2\|_{W^{1,\infty}}$. The main difficulty when one studies the limit $t\to 0^+$ is that the corresponding estimate includes terms behaving as $O(t^{-1})$. The key argument in the proof of stability is the control of  $\|w^2\|_{W^{1,\infty}}$ using the $\q R_\alpha$ norm, which does not depend on time.
\end{nb}

\begin{nb}
	In Proposition \ref{prop:blowupgivenprofile}, we built solutions defined on $(0,1]$ which blow up at $t=0$. One may then ask if these solutions are also stable under $\q R_\alpha$-perturbations. However, in the above proof, the nullity of $(\mathcal{I}S)_x$ is essential in order to close the estimate for $\mathcal{I}w$.
\end{nb}

\appendix

\section{Asymptotic development of the nonlinearity}

\begin{proof}[Proof of Lemma \ref{lem:desenvolveoscilatorio}]
The goal is to obtain the right asymptotics for 
\begin{equation}
\cal N[u](t,p)=p^3\iint_{q_1+q_2+q_3=1} e^{-itp^3(1-q_1^3-q_2^3-q_3^3)}f(t,pq_1)g(t,pq_2)h(t,pq_3)dq_1dq_2
\end{equation}
assuming that $f$ (and \emph{mutatis mutandis} $g,h$) satisfies
\[ 
f(t,p)=\overline{f(t,-p)},\quad \|f\|_{L^\infty} + t^{-\frac{1}{6}}\|\partial_p f\|_{L^2((0,+\infty))} = \|f\|<\infty.
\]
The usual stationary phase arguments either use high regularity assumptions or that all the functions involved have enough spatial decay (specifically $L^2$, in order to apply Parseval's identity), which fail in our setting: as mentioned earlier, the way the computations are performed is critical in order to close the argument with suitable bounds. Before we proceed, let us first explain the ideas of the computations, and start with the main order term. The phase 
\[
Q:= -(1-q_1^3-q_2^3-q_3^3)
\] has four stationary points:
\[
\left(\frac{1}{3},\frac{1}{3},\frac{1}{3}\right),\quad (1,1,-1),\quad (-1,1,1),\quad (1,-1,1).
\]
The last three are connected through the symmetry between $q_1,q_2$ and $q_3$. Then we split the domain of integration using three cutoff function $\phi_j$ with $\phi_1+\phi_2+\phi_3=1$ such that the support of $\phi_j$ does not include the stationary points with $q_k=-1$, $k\neq j$. For example, one may choose
\[
\phi_3\equiv 1\text{ if }q_2>1/6 \text{ and }q_3<1/2,\quad \phi_3\equiv 0\text{ if }q_2<1/12\text{ or }q_3>2/3
\]
and define $\phi_1$, $\phi_2$ in an analogous way. Therefore, without loss of generality,
 in order to study the asymptotics for $\cal N$, it suffices to consider
\[
p^3\iint_{q_1+q_2+q_3=1} e^{itp^3Q}f(t,pq_1)g(t,pq_2)h(t,pq_3)\phi(q_1,q_2)dq_1dq_2,
\]
were $\phi:=\phi_3$ and the relevant stationary points are
\[
\left(\frac{1}{3},\frac{1}{3},\frac{1}{3}\right),\quad (1,1,-1).
\]
If the general stationary phase argument was applicable, then
\begin{align*}
\cal N(t,p)&=k_1(t,p)e^{-\frac{8itp^3}{9}}f\left(t,\frac{p}{3}\right)g\left(t,\frac{p}{3}\right)h\left(t,\frac{p}{3}\right) + k_2(t,p) f(t,p)g(t,p)\overline{h(t,p)} \\&\qquad+ \text{remainder.}
\end{align*}
If one takes smooth cutoff functions around the stationary points $\psi_{1/3}$ and $\psi_1$, then the stationary phase argument for smooth functions (see, for example, \cite{F89}) implies that, up to a small remainder,
\begin{align*}
\MoveEqLeft \frac{1}{3}k_1(t,p)e^{-\frac{8itp^3}{9}}f\left(t,\frac{p}{3}\right)g\left(t,\frac{p}{3}\right)h\left(t,\frac{p}{3}\right)  \\
& = f\left(t,\frac{p}{3}\right)g\left(t,\frac{p}{3}\right)h\left(t,\frac{p}{3}\right) \left(p^3\iint_{q_1+q_2+q_3=1}e^{itp^3Q} \psi_{1/3}(q_1,q_2)\phi(q_1,q_2) dq_1dq_2\right),
\end{align*}
and
\begin{align*}
\MoveEqLeft k_2(t,p)f(t,p)g(t,p)\overline{h(t,p)} \\
& =  f(t,p)g(t,p)\overline{h(t,p)} \left(p^3\iint_{q_1+q_2+q_3=1}e^{itp^3Q} \psi_{1}(q_1,q_2)\phi(q_1,q_2) dq_1dq_2\right).
\end{align*}

This implies that the remainder in the stationary phase argument is given by
\begin{equation}\label{eq:resto}
p^3\iint_{q_1+q_2+q_3=1} e^{itp^3Q}\Phi dq_1dq_2
\end{equation}
where
\begin{align*}
\Phi &=f(t,pq_1)g(t,pq_2)h(t,pq_3) - f(t,p)g(t,p)\overline{h(t,p)}\psi_{1}(q_1,q_2) \\&\qquad- f\left(t,\frac{p}{3}\right)g\left(t,\frac{p}{3}\right)h\left(t,\frac{p}{3}\right)\psi_{1/3}(q_1,q_2).
\end{align*}
We now explain the main ideas behind the computations for the remainder.

The function $\Phi$, due to the fact that $f,g$ and $h$ are Hölder continuous of degree $1/2$, satisfies
\[ |\Phi|\lesssim p^{1/2}t^{1/6}(\sqrt{|q_1-1/3|} + \sqrt{|q_2-1/3|}) \]
and
\[ |\Phi|\lesssim p^{1/2}t^{1/6}(\sqrt{|q_1-1|} + \sqrt{|q_2-1|}). \]
There are three regions of integration: 
\begin{itemize}
\item the inner region: $q_3>1/6$,
\item the middle region: $-2<q_3<1/5$,
\item the outer region: $q_3<-3/2$.
\end{itemize}
In the first and the second regions, we are close to a stationary point and we shall use the Hölder estimate. In the third region, $q_1$ and $q_2$ are large, meaning that we are far from any possible singularity coming from integration by parts. The splitting of the integral into these three regions can be accomplished by using appropriate cut-off functions; however, to simplify the exposition of the proof, we omit these terms.

\bigskip

Instead of applying the usual relation $itp^3(\partial Q)e^{itp^3Q}=\partial(e^{itp^3Q})$, we use
\begin{equation}
	Ke^{itp^3Q}=\frac{1}{1+itp^3K\partial Q}\partial(Ke^{itp^3Q}),\quad \partial K\equiv 1,\quad K=0\text{ at the stationary point}.
\end{equation}
	The introduction of $K$ leads to some simplifications: firstly, there is no singularity appearing in the integration by parts; second, the $K$ in the numerator will add some degeneracy.

\bigskip

The required decay has to come from two integration by parts (one integration eliminates the $p^3$ factor but does not show decay). This has to be done carefully, since $f,g$ and $h$ cannot be differentiated more than once. The key fact is that one may differentiate, for example, 
\[ f_p(pq_1)g(pq_2)h(pq_3) \]
 in the $q_2$ (or $q_3$) direction. Therefore, the two required integration by parts are made in different directions, so that no second derivatives of $f$ appear.

Even though one could perform all the computations in the $q_1,q_2$ coordinates, we introduce some linear change of variables so that it becomes clearer in which direction we integrate by parts and which terms are irrelevant in each region. For example, we shall say that $q_1$ is irrelevant on the middle region and throw it away when taking absolute values in the integrand.

\bigskip

We now bound the remainder terms in detail. Throughout this proof, $\tau=tp^3$ and $p>0$. 
	Consider the change of variables
\[
	1-q_1=\lambda- \mu ,\quad 1-q_2= \lambda +  \mu ,\quad 1-q_3=2(1-\lambda).
\]
	Notice that both stationary points satisfy $ \mu =0$. We now use the relation
\[
	e^{i\tau Q}=\frac{1}{1+4i\tau  \mu ^2(1-\lambda)}\partial_ \mu ( \mu e^{i\tau Q})
\]
	and integrate by parts \eqref{eq:resto}:
\begin{align*}
\iint e^{i\tau Q}\Phi dq_1dq_2 &= \iint e^{i\tau Q}\Phi \frac{8i\tau  \mu ^2(1-\lambda)}{(1+4i\tau  \mu ^2(1-\lambda))^2}dq_1dq_2 \\
& \quad + \iint e^{i\tau Q}\Phi_{q_1}\frac{ \mu }{1+4i\tau  \mu ^2(1-\lambda)}dq_1dq_2 \\
& \quad - \iint e^{i\tau Q}\Phi_{q_2}\frac{ \mu }{1+4i\tau  \mu ^2(1-\lambda)}dq_1dq_2 \\
& = M_1 + M_2 - M_3.
\end{align*}
The estimate for $M_3$ follows from similar computations as those for $M_2$. 

We will bound $M_1$ and $M_2$ in separately, and depending whether $\tau$ is less or greater than 1,  in the four claims below.

Let us focus first of $M_1$. We take $\eta= \mu \sqrt{1-\lambda}$ and use, for a fixed $\lambda_0\in\{0,2/3\}$, we ahve
\[
	e^{i\tau Q}=\frac{1}{1 + 2i\tau (\lambda-\lambda_0)\lambda(2-3\lambda)}\partial_\lambda((\lambda-\lambda_0)e^{i\tau Q}).
\]
Hence
\begin{align*}
M_1 & = \iint e^{i\tau Q}\Phi \frac{8i\tau  \mu ^2(1-\lambda)}{(1+4i\tau  \mu ^2(1-\lambda))^2}dq_1dq_2= \iint e^{i\tau Q} \Phi \frac{8i\tau \eta^2}{(1+4i\tau \eta^2)^2} \frac{d\eta d\lambda}{\sqrt{1-\lambda}} \\
& = \iint e^{i\tau Q}\left[ \Phi_{q_1}\left(1+\frac{ \mu }{2(1-\lambda)}\right) + \Phi_{q_2}\left(1-\frac{ \mu }{2(1-\lambda)}\right) -2\Phi_{q_3}\right] \\&\qquad\qquad\times\frac{8i\tau (\lambda-\lambda_0) \eta^2d\lambda d\eta}{(1+4i\tau\eta^2)^2(1-\lambda)^{1/2}(1 + 2i\tau (\lambda-\lambda_0)\lambda(2-3\lambda))} \\
& \quad + \iint e^{i\tau Q}\Phi \left(-\frac{1}{2\sqrt{(1-\lambda)}} + \frac{2i\tau\left( \lambda(2-3\lambda) + (\lambda-\lambda_0)(2-3\lambda) - 3(\lambda-\lambda_0)\lambda \right)}{1+3i\tau(\lambda-\lambda_0)\lambda(2-3\lambda)}\right) \\
& \qquad \times\frac{8i\tau (\lambda-\lambda_0) \eta^2d\lambda d\eta}{(1+4i\tau\eta^2)^2(1-\lambda)^{1/2}(1 + 2i\tau (\lambda-\lambda_0)\lambda(2-3\lambda))}.
\end{align*}

\begin{claim}
For $\tau \ge 1$, we have the bound on $M_1$:
\[ |M_1| \lesssim \tau^{-13/12}. \]
\end{claim}

\begin{proof}
Bounds in the inner region: here we choose $\lambda_0=0$. Since 
\[	2-3\lambda, \ 1-\lambda, \ 1\pm \frac{ \mu }{2(1-\lambda)}\text{ are irrelevant,} \]
a direct bound on $M_1$ yields
\begin{align*}
|M_1| & \lesssim \iint \frac{|\tau \eta^2 \lambda|}{|1+i\tau \eta^2|^2|1+i\tau \lambda^2|}(|\Phi_{q_1}| + |\Phi_{q_2}| + |\Phi_{q_3}|)d\lambda d\eta \\
& \quad + \iint |\Phi|\left(\frac{|\tau \lambda\eta^2|}{|1+i\tau \lambda^2||1+i\tau \eta^2|^2} + \frac{|\tau \lambda^2||\tau\eta^2|}{|1+i\tau \lambda^2|^2|1+i\tau \eta^2|^2}\right)d\lambda d\eta.
\end{align*}
The first term is bounded by Cauchy-Schwarz:
\begin{align*}
\MoveEqLeft \iint \frac{|\tau \eta^2 \lambda|}{|1+i\tau \eta^2|^2|1+i\tau \lambda^2|}(|\Phi_{q_1}| + |\Phi_{q_2}| + |\Phi_{q_3}|)d\lambda d\eta \\
&\lesssim  \tau^{1/2} \int \frac{\tau\eta^2d\eta}{|1+i\tau\eta^2|^2}\left(\int \frac{\lambda^2d\lambda}{|1+i\tau \lambda^2|^2}\right)^{1/2} \\
&  \lesssim \tau^{-13/12}.
\end{align*}
For the second and the third term, we use the Hölder estimate for $|\Phi|$:
\[ |\Phi|\lesssim \tau^{1/6}(\sqrt{|\lambda|} + \sqrt{|\eta|}) \]
and obtain
\begin{align*}
\iint |\Phi|\frac{|\tau \lambda\eta^2| d\lambda d\eta}{|1+i\tau \lambda^2||1+i\tau \eta^2|^2} \lesssim \tau^{1/6}\iint \frac{|\lambda |^{3/2}|\tau\eta^2| + |\tau \lambda ||\eta|^{5/2}}{|1+i\tau \lambda ^2||1+i\tau\eta^2|^2}d\lambda d\eta \lesssim \tau^{-13/12}
\end{align*}
and
\begin{align*}
\iint |\Phi|\frac{|\tau \lambda ^2||\tau\eta^2| d\lambda d\eta}{|1+3i\tau \lambda ^2|^2|1+i\tau \eta^2|^2} \lesssim \tau^{1/6}\iint \frac{|\tau|^2(|\lambda |^{5/2}|\eta|^2 + |\lambda |^2|\eta|^{5/2})}{|1+i\tau \lambda ^2|^2|1+i\tau \eta^2|^2}d\lambda d\eta \lesssim {\tau^{-13/12}}.
\end{align*}
	
Bounds in the middle region: here we take $\lambda _0=2/3$. Since
\[ \lambda , 1-\lambda , 1\pm \frac{ \mu }{2(1-\lambda )}\text{ are irrelevant,} \]
a direct bound on $M_1$ yields
\begin{align*}
|M_1|&\lesssim \iint \frac{|\tau \eta^2 (\lambda -2/3)|}{|1+i\tau \eta^2|^2|1+i\tau (\lambda -2/3)^2|}(|\Phi_{q_1}| + |\Phi_{q_2}| + |\Phi_{q_3}|)d\lambda d\eta \\&+ \iint |\Phi|\left(\frac{|\tau (\lambda -2/3)\eta^2|}{|1+i\tau (\lambda -2/3)^2||1+i\tau \eta^2|^2} + \frac{|\tau (\lambda -2/3)^2||\tau\eta^2|}{|1+i\tau (\lambda -2/3)^2|^2|1+i\tau \eta^2|^2}\right)d\lambda d\eta.
	\end{align*}
	and the estimate follows as in the inner region.
	
Bounds in the outer region: we consider $\lambda _0=0$ and use
\[
	1\pm \frac{ \mu }{2(1-\lambda )}\sim 1\pm \frac{\eta}{|\lambda |^{3/2}},\quad 1+3\lambda \sim \lambda ,\quad 1+3i\tau \lambda ^2\sim \tau \lambda ^2
\]
	to obtain
	\begin{align*}
	|M_1|\lesssim \iint \frac{|\lambda \eta^2|}{|1+i\tau \eta^2|^2|\lambda |^{7/2}}\left(\left(1+\frac{|\eta|}{|\lambda |^{3/2}}\right)(|\Phi_{q_1}| + |\Phi_{q_2}|) + |\Phi_{q_3}| + \frac{1}{|\lambda |^{1/2}}|\Phi|\right)\lesssim \tau^{-13/12},
	\end{align*}
	where the terms with $|\eta|/|\lambda |^{3/2}$ is taken care with Cauchy-Schwarz in the $\eta$-variable.
\end{proof}

\begin{claim} \label{claim:M1_tau<1}
For $\tau <1$, we have the bound on $M_1$
\[ |M_1| \lesssim \tau^{-5/6}. \]
\end{claim}

\begin{proof}
In the outer region, we take $\lambda _0=0$ and estimate
\begin{align*}
|M_1| & \lesssim \iint \frac{|\lambda \eta^2|}{|1+i\tau \eta^2|^2|\lambda |^{1/2}|1+i\tau \lambda ^3|}\left(\left(1+\frac{|\eta|}{|\lambda |^{3/2}}\right)(|\Phi_{q_1}| + |\Phi_{q_2}|) + |\Phi_{q_3}| + \frac{1}{|\lambda |^{1/2}}|\Phi|\right) \\
& \lesssim \tau^{-5/6}.
\end{align*}
In the inner and middle regions, we simply use the fact that $\Phi$ is bounded and conclude to a better bound than needed:
\[ \left|\iint e^{itp^3Q}\Phi dq_1dq_2\right| \lesssim 1. \qedhere \]
\end{proof}

\begin{claim} \label{claim:M2_tau>1}
For $\tau \ge 1$, we have the bound on $M_2$:
\[ |M_2| \lesssim \tau^{-13/12}. \]
\end{claim}

\begin{proof}
We consider the change of variables
\[ \mu =\frac{3\zeta+\xi-2}{2},\quad \lambda =\frac{2-\zeta+\xi}{2}. \]
One may obtain this transformation by going back to the $\xi$ variables, switching $q_1$ with $q_3$ and then redoing the $\lambda,\mu$ transformation. In this way, $q_1$ depends on a single variable $\zeta$.

In this coordinate system, the stationary points are
\[	(\xi,\zeta)=(0,2/3) \text{ and } (-1,1). \]
In the inner region, we use the relation
\[ e^{i\tau Q}=\frac{1}{1+4i\tau \xi^2(1-\zeta)}\partial_\xi (\xi e^{i\tau Q}). \]
Define
\[ A=1+4i\tau \xi^2(1-\zeta),\quad B=1+4i\tau  \mu ^2(1-\lambda )=1+i\tau(3\zeta+\xi-2)^2(\zeta-\xi)/2. \]
We now integrate by parts:
\begin{align*}
M_2 & = \iint e^{i\tau Q}\frac{ \mu \Phi_{q_1}}{B}d\xi d\zeta \\&= \iint \xi e^{i\tau Q}\left[\Phi_{q_1q_2} \frac{ \mu }{AB} + \frac{\Phi_{q_1}}{AB} +  \mu \Phi_{q_1}\left(\frac{A_\xi}{A^2B} + \frac{B_\xi}{AB^2}\right)\right]d\xi d\zeta.
\end{align*}
Since, in this region,
\[ 1\ge \zeta, \quad 1-z\sim 1,\quad  \left|\frac{\xi A_\xi}{A}\right|\lesssim 1,\quad B\sim 1+i\tau  \mu ^2, \quad B_\xi \sim \tau  \mu, \]
we can now estimate $M_2$ as follows:
\begin{align*}
|M_2| & \lesssim \iint \frac{|\xi  \mu |||\Phi_{q_1q_2}| + |\xi||\Phi_{q_1}|}{|1+i\tau \xi^2||1+i\tau  \mu ^2|} \\
& \quad + | \mu ||\Phi_{q_1}|\left(\frac{1}{|1+i\tau \xi^2||1+i\tau  \mu ^2|} + \frac{|\tau  \mu \xi|}{|1+i\tau \xi^2||1+i\tau  \mu ^2|^2}\right)d\xi d \mu.
\end{align*}
We apply Cauchy-Schwarz to all four terms:
\begin{align*}
\iint \frac{|\xi  \mu |||\Phi_{q_1q_2}|}{|1+i\tau \xi^2||1+i\tau  \mu ^2|}d\xi d \mu & \lesssim \tau^{1/3}\left(\int \frac{| \mu |^2d \mu }{|1+i\tau  \mu ^2|^2}\right)^{\frac{1}{2}}\left(\int \frac{|\xi|^2d\xi}{|1+i\tau \xi^2|^2}\right)^{\frac{1}{2}} \lesssim \tau^{-7/6}, \\
\iint \frac{|\xi ||\Phi_{q_1}|}{|1+i\tau \xi^2||1+i\tau  \mu ^2|}d\xi d \mu & \lesssim \tau^{1/6}\int \frac{d \mu }{|1+i\tau  \mu ^2|}\left(\int \frac{|\xi|^2d\xi}{|1+i\tau \xi^2|^2}\right)^{\frac{1}{2}} \lesssim \tau^{-13/12}, \\
\iint \frac{| \mu  ||\Phi_{q_1}|}{|1+i\tau \xi^2||1+i\tau  \mu ^2|}d\xi d \mu & \lesssim \tau^{1/6}\left(\int\frac{| \mu |^2d \mu }{|1+i\tau  \mu ^2|^2}\right)^{\frac{1}{2}}\int \frac{d\xi}{|1+i\tau\xi^2|}\lesssim \tau^{-13/12}, \\
\iint \frac{| \tau  \mu ^2\xi||\Phi_{q_1}|}{|1+i\tau \xi^2||1+i\tau  \mu ^2|^2}d\xi d \mu & \lesssim \tau^{1/6}\left(\int \frac{|\xi|^2d\xi}{|1+i\tau\xi^2|^2}\right)^{\frac{1}{2}}\int \frac{|\tau  \mu ^2|d \mu }{|1+i\tau  \mu ^2|^2} \lesssim \tau^{-13/12}, 
\end{align*}
which gives suitable bounds.

In the middle region, we use the following relation to capture the stationary point: 
\[ e^{i\tau Q}=\frac{1}{1+2i\tau \xi(\xi+1)(1-\zeta)}\partial_\xi ((\xi+1) e^{i\tau Q}). \]
After an integration by parts, the computations are similar to those of the inner region and are left to the reader.
	
In the outer region, we consider the case $|1-\zeta|<1/100$ first. Then
\[ \mu \sim \xi,\quad 1-\lambda \sim \xi,\quad |\xi|\gtrsim 1 \]
and so
\begin{align*}
|M_2| &\lesssim \iint \frac{\xi}{|1+i\tau\xi^2\zeta||1+i\tau \xi^3|}\left[\xi\Phi_{q_1q_2} + \Phi_{q_1} + \xi \Phi_{q_1}\left(\frac{|\tau\xi\zeta|}{|1+i\tau\xi^2\zeta|} + \frac{|\tau \xi^3|}{|1+i\tau \xi^3|}\right)\right]d\xi d\zeta.
\end{align*}
We now estimate each term separately.
\begin{align*}
\iint \frac{|\xi^2\Phi_{q_1q_2}|}{|1+i\tau\xi^2\zeta||1+i\tau \xi^3|}d\xi d\zeta & \lesssim \tau^{1/3}\left( \iint \frac{\xi^4d\xi d\zeta}{|1+i\tau \xi^2 \zeta|^2|1+i\tau \xi^3|^2} \right)^{\frac{1}{2}}\\
& \lesssim \tau^{-2/3}\left(\int_{|\xi|\gtrsim \tau^{1/3}} \frac{\xi^2d\xi}{|1+i\xi^3|^2} \int\frac{d\zeta}{|1+i\zeta|^2} \right)^{\frac{1}{2}} \lesssim \tau^{-7/6}, \\
\iint \frac{|\xi||\Phi_{q_1}|}{|1+i\tau\xi^2\zeta||1+i\tau \xi^3|}d\xi d\zeta & \lesssim \tau^{1/6}\int \frac{d\xi}{|1+i\tau \xi^3|}\left(\int \frac{\xi^2d\zeta}{|1+i\tau \xi^2\zeta|^2}\right)^{\frac{1}{2}} \lesssim \tau^{-7/6}, \\
\iint \frac{\xi^2|\Phi_{q_1}|}{|1+i\tau\xi^2\zeta||1+i\tau \xi^3|}d\xi d\zeta & \lesssim\tau^{1/6} \int \frac{|\xi|d\xi}{|1+i\tau \xi^3|}\left(\int \frac{\xi^2d\zeta}{|1+i\tau \xi^2\zeta|^2}\right)^{\frac{1}{2}}\lesssim \tau^{-7/6}.
\end{align*}
In the remaining case where $|1-\zeta|>1/100$, we further split in the cases $| \mu |<1/100$ and $| \mu |>1/100$. In the former case,
\[ \zeta \sim -\xi,\quad 1-z\sim -\xi, \quad |\tau \xi^2 \zeta| \ge 1 \]
meaning that
\[ |M_2| \lesssim \iint \frac{|\xi|}{|\tau \xi^3||1+i\tau  \mu ^2\xi|}\left[| \mu \Phi_{q_1q_2}| + |\Phi_{q_1}| + | \mu \Phi_{q_1}|\left(1+\frac{|\tau  \mu \xi|}{|1+i\tau  \mu ^2\xi|}\right)\right]d\xi d \mu. \]
Again using Cauchy-Schwarz inequality,
\begin{align*}
\iint \frac{|\xi|}{|\tau \xi^3||1+i\tau  \mu ^2\xi|}| \mu \Phi_{q_1q_2}|d\xi d \mu & \lesssim \tau^{-2/3}\left(\iint \frac{ \mu ^2d \mu d\xi}{\xi^4|1+i\tau  \mu ^2\xi|^2}\right)^{\frac{1}{2}} \lesssim \tau^{-17/12}, \\
\iint  \frac{|\xi|}{|\tau \xi^3||1+i\tau  \mu ^2\xi|}|\Phi_{q_1}|d\xi d \mu & \lesssim \tau^{-5/6}\int \frac{d\xi}{\xi^2}\left(\int \frac{d \mu }{|1+i\tau  \mu ^2\xi|}\right)^{\frac{1}{2}}\lesssim \tau^{-13/12}, \\
\iint \frac{| \mu \xi\Phi_{q_1}|}{|\tau \xi^3||1+i\tau  \mu ^2\xi|}\left(1+\frac{|\tau  \mu  \xi|}{|1+i\tau  \mu ^2\xi|}\right)d\xi d \mu & \lesssim \iint  \frac{|\xi| |\Phi_{q_1}|}{|\tau \xi^3||1+i\tau  \mu ^2\xi|}d\xi d \mu  \lesssim\tau^{-13/12}.
\end{align*}
Finally, in the latter case where $| \mu |>1/100$, one has
\[ |1+i\tau \xi^2\zeta|\gtrsim |\tau\xi^2 \zeta|,\quad |1+i\tau  \mu ^2(1-\lambda )|\gtrsim|\tau  \mu ^2(1-\lambda )|. \]
This implies directly the even better bound $|M_2| \lesssim \tau^{-5/3}$.
\end{proof}

\begin{claim}
For $\tau <1$, we have the bound on $M_2$:
\[ |M_2| \lesssim \tau^{-2/3}. \]
\end{claim}

\begin{proof}
In the outer region, we further split depending on the sizes of $\zeta$ and $\mu$. Assume first that $|\zeta|<1/100$. Proceeding as in Claim \ref{claim:M2_tau>1}, we obtain
\[ |M_2| \lesssim \tau^{-2/3}. \]
If $|\zeta|>1/100$ and $| \mu |<1/100$, then
\[ \zeta\sim -\xi,\quad 1-\lambda \sim -\xi. \]
This yields the estimates
\begin{align*}
\iint \frac{| \mu \xi|}{|1+i\tau\xi^3||1+i\tau  \mu ^2\xi|}|\Phi_{q_1q_2}|d\xi d \mu & \lesssim \tau^{1/3}\left(\iint \frac{\xi^2  \mu ^2 d\xi d \mu }{|1+i\tau \xi^3|^2|1+i\tau  \mu ^2\xi|^2}\right)^{\frac{1}{2}} \lesssim \tau^{-2/3}, \\
\iint \frac{|\xi|}{|1+i\tau\xi^3||1+i\tau  \mu ^2\xi|}|\Phi_{q_1}|d\xi d \mu & \lesssim \tau^{1/6}\int \frac{d\xi}{|1+i\tau\xi^3|}\left(\int \frac{\xi^2 d \mu }{|1+i\tau  \mu ^2\xi|^2}\right)^{\frac{1}{2}} \lesssim \tau^{-2/3}.
\end{align*}
Finally, if $| \mu |>1/100$, then
\[ \mu \sim \zeta + \xi,\quad 1 - \lambda \sim \zeta - \xi, \]
and so
\begin{align*}
\MoveEqLeft \iint \frac{(|\zeta|+|\xi|)|\xi|}{|1+i\tau \xi^2\zeta||1+i\tau (\zeta+\xi)^2(\zeta-\xi)|}|\Phi_{q_1q_2}|d\zeta d\xi \\ 
& \lesssim \tau^{1/3}\left( \iint \frac{(|\zeta|^2+|\xi|^2)|\xi|^2}{|1+i\tau \xi^2\zeta|^2|1+i\tau (\zeta+\xi)^2(\zeta-\xi)|^2}|d\zeta d\xi\right)^{\frac{1}{2}}\lesssim \tau^{-2/3}, \\
\MoveEqLeft \iint \frac{|\xi|}{|1+i\tau \xi^2\zeta||1+i\tau (\zeta+\xi)^2(\zeta-\xi)|}|\Phi_{q_1}|d\zeta d\xi \\
& \lesssim \tau^{1/6}\int d\xi \left(\frac{|\xi|^2d\zeta}{|1+i\tau \xi^2\zeta|^2|1+i\tau (\zeta+\xi)^2(\zeta-\xi)|^2} \right)^{\frac{1}{2}}\lesssim \tau^{-2/3}.
\end{align*}

Finally, in the inner and middle regions, we proceed as in Claim \ref{claim:M1_tau<1}: the fact that $\Phi$ is bounded yields the better bound
\[ 
	\left|\iint e^{itp^3Q}\Phi dq_1dq_2\right| \lesssim 1. \qedhere
\]
\end{proof}

Gathering all these bounds, and taking into account the main order term, gives the expansion \eqref{eq:N_expansion} together with the bound \eqref{eq:N_remainder} on the remainder. The proof is complete.	
\end{proof}

\nocite{HM80}
\nocite{HN01}
\nocite{HN99}
\nocite{CV08}
\nocite{DZ95}
\nocite{GPR16}
\nocite{HaGr16}
\nocite{PV07}
\nocite{FIKN06}
\nocite{FA83}

\bibliography{biblio_mkdv}
\bibliographystyle{plain}

\bigskip
\bigskip

\normalsize

\begin{center}
{\scshape Simão Correia}\\
{\footnotesize
	CMAF-CIO, Universidade de Lisboa\\
Edif\'{\i}cio C6, Campo Grande\\
1749-016 Lisboa, Portugal\\
\email{sfcorreia@fc.ul.pt}
}

\bigskip

{\scshape Raphaël Côte}\\
{\footnotesize
Université de Strasbourg\\
CNRS, IRMA UMR 7501\\
F-67000 Strasbourg, France\\
\email{cote@math.unistra.fr}
}

\bigskip

{\scshape Luis Vega}\\
{\footnotesize
Departamento de Matemáticas\\
Universidad del País Vasco UPV/EHU\\
Apartado 644, 48080, Bilbao, Spain\\
\ \\
Basque Center for Applied Mathematics BCAM\\
Alameda de Mazarredo 14, 48009 Bilbao, Spain\\
\email{lvega@bcamath.org}
}
\end{center}

\end{document}